\documentclass[10pt]{amsart} 

\IfFileExists{srcltx.sty}{\usepackage{srcltx}}

\usepackage{comment}
\usepackage{graphicx}
\usepackage[utf8]{inputenc}
\usepackage{xspace,amssymb,amsfonts,euscript}
\usepackage{amsthm,amsmath}
\usepackage{palatino}
\usepackage{euscript}
\usepackage{tikz-cd}
\usepackage{shuffle}
\usepackage{url}
\input xy \xyoption {all}

\numberwithin{equation}{section}

\RequirePackage{color}
\definecolor{myred}{rgb}{0.75,0,0}
\definecolor{mygreen}{rgb}{0,0.5,0}
\definecolor{myblue}{rgb}{0,0,0.65}

\RequirePackage{ifpdf}
\ifpdf
  \IfFileExists{pdfsync.sty}{\RequirePackage{pdfsync}}{}
  \RequirePackage[pdftex,
   colorlinks = true,
   urlcolor = myblue, 
   citecolor = mygreen, 
   linkcolor = myred, 
   pagebackref,
   bookmarksopen=true]{hyperref}
\else
  \RequirePackage[hypertex]{hyperref}
\fi

\RequirePackage{ae, aecompl, aeguill} 

    \def\CM{{\mathbb{C}}}

    \def\NM{{\mathbb{N}}}

    \def\ZM{{\mathbb{Z}}}

\newcommand{\km}{\Bbbk}

    \def\AC{{\mathcal{A}}}

    \def\DC{{\mathcal{D}}}
    \def\EC{{\mathcal{E}}}
    \def\FC{{\mathcal{F}}}
    
    \def\HC{{\mathcal{H}}}
    \def\IC{{\mathcal{I}}}

    \def\LC{{\mathcal{L}}}

    \def\OC{{\mathcal{O}}}
    \def\PC{{\mathcal{P}}}

    \def\SC{{\mathcal{S}}}

\renewcommand\l{\lambda}

\def\a{\alpha}

\def\G{\Gamma}

\def\D{\Delta}
\def\e{\varepsilon}

\def\L{\Lambda}

\def\r{\rho}

\newcommand{\nc}{\newcommand}

\def\to{\rightarrow}

\nc{\ic}{\mathbf{IC}}

\nc{\gl}{{\mathfrak{gl}}}

\DeclareMathOperator{\Hom}{Hom}

\DeclareMathOperator{\End}{End} 

\DeclareMathOperator{\id}{Id}

\DeclareMathOperator{\im}{im}

\DeclareMathOperator{\Rad}{rad}

\newtheorem{thm}{Theorem}[section] 
\newtheorem{lemma}[thm]{Lemma}

\newtheorem{prop}[thm]{Proposition}
\newtheorem{cor}[thm]{Corollary}

\newtheorem{manualtheoreminner}{Theorem}
\newenvironment{manualtheorem}[1]{%
  \renewcommand\themanualtheoreminner{#1}%
  \manualtheoreminner
}{\endmanualtheoreminner}

\newtheorem{manualcorollaryinner}{Corollary}
\newenvironment{manualcorollary}[1]{%
  \renewcommand\themanualcorollaryinner{#1}%
  \manualcorollaryinner
}{\endmanualcorollaryinner}

\theoremstyle{definition}
\newtheorem{defn}[thm]{Definition}

\newtheorem{ex}[thm]{Example}
\newtheorem{remark}[thm]{Remark}

\newtheorem{manualdefinitioninner}{Definition}
\newenvironment{manualdefinition}[1]{%
  \renewcommand\themanualdefinitioninner{#1}%
  \manualdefinitioninner
}{\endmanualdefinitioninner}

\DeclareMathOperator{\Ext}{Ext}

\DeclareMathOperator{\cok}{cok}

\newcommand{\into}{\hookrightarrow}

\def\Mod{\text{-}\mathrm{Mod}}
\def\rMod{\mathrm{Mod}\text{-}}

\def\mod{\text{-}\mathrm{mod}}
\def\rmod{\mathrm{mod}\text{-}}

\DeclareMathOperator{\Loc}{Loc}

\def\gl{\mathfrak{gl}}
\def\ic{\mathbf{IC}}

\def\multiset#1#2{\ensuremath{\left(\kern-.3em\left(\genfrac{}{}{0pt}{}{#1}{#2}\right)\kern-.3em\right)}}

\RequirePackage{tikz-cd}
\RequirePackage{amssymb}
\usetikzlibrary{calc}
\usetikzlibrary{decorations.pathmorphing}

\tikzset{curve/.style={settings={#1},to path={(\tikztostart)
    .. controls ($(\tikztostart)!\pv{pos}!(\tikztotarget)!\pv{height}!270:(\tikztotarget)$)
    and ($(\tikztostart)!1-\pv{pos}!(\tikztotarget)!\pv{height}!270:(\tikztotarget)$)
    .. (\tikztotarget)\tikztonodes}},
    settings/.code={\tikzset{quiver/.cd,#1}
        \def\pv##1{\pgfkeysvalueof{/tikz/quiver/##1}}},
    quiver/.cd,pos/.initial=0.35,height/.initial=0}

\tikzset{tail reversed/.code={\pgfsetarrowsstart{tikzcd to}}}
\tikzset{2tail/.code={\pgfsetarrowsstart{Implies[reversed]}}}
\tikzset{2tail reversed/.code={\pgfsetarrowsstart{Implies}}}
\tikzset{no body/.style={/tikz/dash pattern=on 0 off 1mm}}

\title[Stratifications of abelian categories]{Stratifications of abelian categories}

\author{Giulian Wiggins}

\begin{document}

\maketitle

\begin{abstract} 
This paper studies abelian categories that can be decomposed into smaller abelian categories via iterated recollements - such a decomposition we call a {\em stratification}. Examples include the categories of (equivariant) perverse sheaves and $\e$-stratified categories (in particular highest weight categories) in the sense of Brundan-Stroppel \cite{BS18}.

We give necessary and sufficient conditions for an abelian category with a stratification to be equivalent to a category of finite dimensional modules of a finite dimensional algebra - this generalizes the main result of Cipriani-Woolf \cite{CW21}. Furthermore, we give necessary and sufficient conditions for such a category to be $\e$-stratified - this generalizes the characterisation of highest weight categories given by Krause \cite{Kra17a}.
\end{abstract}

~~
\\
~~
\\
~


\tableofcontents

\newpage

\section{Introduction}\label{intro}

A {\em recollement} of abelian categories is a short exact sequence of abelian categories
\[
\begin{tikzcd}
	{0} & {\AC_Z} & \AC & {\AC_U} & {0}
	\arrow[from=1-1, to=1-2]
	\arrow["{i_*}", from=1-2, to=1-3]
	\arrow["{j^*}", from=1-3, to=1-4]
	\arrow[from=1-4, to=1-5]
\end{tikzcd}
\]
in which $\AC_Z$ is a Serre subcategory of $\AC$ (with Serre quotient $\AC_U$), $i_*$ has both a left and right adjoint, and $j^*$ has fully-faithful left and right adjoints. 
We usually denote such a recollement of abelian categories by the  diagram
\[
\begin{tikzcd}
	{\AC_Z} && \AC && {\AC_U}
	\arrow["{i_*}", from=1-1, to=1-3]
	\arrow["{i^{*}}"', shift right=5, from=1-3, to=1-1]
	\arrow["{i^{!}}", shift left=5, from=1-3, to=1-1]
	\arrow["{j^*}", from=1-3, to=1-5]
	\arrow["{j_{!}}"', shift right=5, from=1-5, to=1-3]
	\arrow["{j_{*}}", shift left=5, from=1-5, to=1-3]
\end{tikzcd}
\]
where $(i^*, i_*, i^!)$ and $(j_!, j^*, j_*)$ are adjoint triples.

This definition is motivated by recollements of triangulated categories as defined by Beilinson, Bernstein and Deligne \cite{BBD82} to generalize  Grothendieck's six functors relating the constructible derived category, $\DC (X, \km)$, of sheaves on a variety $X$ (with coefficients in a field $\km$) and the constructible derived categories, $\DC (Z, \km)$ and $\DC(U, \km)$, of sheaves on a closed subvariety $Z \subset X$ and open complement $U := X \backslash Z$ i.e. the situation:
\[
\begin{tikzcd}
	{\DC (Z, \km)} && \DC (X, \km) && {\DC (U, \km)}
	\arrow["{i_{*}}", from=1-1, to=1-3]
	\arrow["{i^{*}}"', shift right=5, from=1-3, to=1-1]
	\arrow["{i^{!}}", shift left=5, from=1-3, to=1-1]
	\arrow["{j^{*}}", from=1-3, to=1-5]
	\arrow["{j_{!}}"', shift right=5, from=1-5, to=1-3]
	\arrow["{j_{*}}", shift left=5, from=1-5, to=1-3]
\end{tikzcd}
\]
Here, $i \colon Z \into X$ is the closed embedding and $j \colon U \into X$ is the complementary open embedding.

Note that if $\SC h (X, \km)$ is the category of sheaves on $X$, and ${\HC}^{i} \colon \DC^{b} (X, \km) \to \SC h (X, \km)$ is the $i$-th cohomology functor, then the following is an example of a recollement of abelian categories.
\[
\begin{tikzcd}
	{\SC h (Z, \km)} && {\SC h (X, \km)} && {\SC h (U, \km)}
	\arrow["{i_{*}}", from=1-1, to=1-3]
	\arrow["{{\HC}^{0} i^{*}}"', shift right=5, from=1-3, to=1-1]
	\arrow["{{\HC}^{0} i^{!}}", shift left=5, from=1-3, to=1-1]
	\arrow["{j^{*}}", from=1-3, to=1-5]
	\arrow["{{\HC}^{0} j_{!}}"', shift right=5, from=1-5, to=1-3]
	\arrow["{{\HC}^{0} j_{*}}", shift left=5, from=1-5, to=1-3]
\end{tikzcd}
\]
More generally, given a recollement of triangulated categories with compatible $t$-structures,
 there is a canonical recollement of abelian categories on the hearts of the $t$-structure \cite[Proposition 1.4.16]{BBD82} defined using the zero-th cohomology functor arising from the $t$-structure (as in the above example).

In representation theory, recollements of abelian categories arise from modules of algebras $A$ with an idempotent $e$. In this situation there is a recollement
\[
\begin{tikzcd}
	{0} & {A / AeA \Mod} & A \Mod & {eAe \Mod} & {0}
	\arrow[from=1-1, to=1-2]
	\arrow["{i_*}", from=1-2, to=1-3]
	\arrow["{j^*}", from=1-3, to=1-4]
	\arrow[from=1-4, to=1-5]
\end{tikzcd}
\]
where $i_* \colon A / AeA \Mod \to A \Mod $ is the restriction functor, and $j^* \colon A \Mod \to eAe \Mod$ is the functor sending an $A$-module $M$ to the $eAe$-module $eM$.

Recollements often arise as part of an iterated collection of recollements, giving a kind of filtration of a larger category by smaller categories. Consider, for example, the category, $P_{\L} (X, \km)$, of perverse sheaves that are constructible with respect to a stratification, $X := \bigcup_{\l \in \L}X_{\l}$, of a variety $X$. 
For each strata $X_{\l}$, there is a recollement of abelian categories
\[
\begin{tikzcd}
	{0} & {P_{\L} (\overline{X_{\l}} \setminus X_{\l}, \km)} & {P_{\L} (\overline{X_{\l}}, \km)} & {\Loc^{ft} (X_{\l}, \km)} & {0}
	\arrow[from=1-1, to=1-2]
	\arrow[from=1-2, to=1-3]
	\arrow[from=1-3, to=1-4]
	\arrow[from=1-4, to=1-5]
\end{tikzcd}
\]
where $\Loc^{ft} (X_{\l}, \km)$ is the category of local systems on $X_{\l}$ of finite type. 
More generally, if $\G \subset \L$ is such that $X_{\G} := \bigcup_{\l \in \G} X_{\l}$ is a closed subvariety of $X$, and $X_{\l}$ is open in $X_{\G}$, then the following is a recollement of abelian categories.
\[
\begin{tikzcd}
	{0} & {P_{\L} (X_{\G \setminus \{\l\} }, \km)} & {P_{\L} (X_{\G}, \km)} & {\Loc^{ft} (X_{\l}, \km)} & {0}
	\arrow[from=1-1, to=1-2]
	\arrow[from=1-2, to=1-3]
	\arrow[from=1-3, to=1-4]
	\arrow[from=1-4, to=1-5]
\end{tikzcd}
\]

Similar iterations of recollement occur in highest weight categories (as defined in \cite{CPS88b}). Indeed, let $\AC$ be a highest weight category with respect to a poset $\L$, and let $\{ \D_{\l} ~|~ \l \in \L \}$ be the collection of standard objects in $\AC$. For each lower\footnote{A subposet $\G \subset \L$ is {\em lower} if for any $\l, \mu \in \L$, $\mu \leq \l$ and $\l \in \G$ implies $\mu \in \G$.} subposet $\G \subset \L$, let $\AC_{\G}$ be the Serre subcategory generated by objects $\{ \D_{\l} ~|~ \l \in \G \}$. For each lower subposet $\G \subset \L$ and  maximal $\l \in \G$ there is a recollement of abelian categories
\[
\begin{tikzcd}
	{0} & {\AC_{\G \setminus \{\l\}}} & {\AC_{\G}} & {\End_{\AC} (\D_{\l})^{op} \Mod} & {0}
	\arrow[from=1-1, to=1-2]
	\arrow[from=1-2, to=1-3]
	\arrow["{j^*}", from=1-3, to=1-4]
	\arrow[,from=1-4, to=1-5]
\end{tikzcd}
\]
where $j^* := \Hom(\D_{\l}, -)\colon \AC_{\G} \to \End_{\AC} (\D_{\l})^{op} \Mod$.

We unify these examples by the following notion.

\begin{manualdefinition}{\ref{stratification}}
A  {\em stratification} of an abelian category $\AC$ by a non-empty poset $\L$ consists of the following data
\begin{enumerate}
\item[(i)]
For each lower subposet $\G \subset \L$, define a Serre subcategory, $\AC_{\G}$, of $\AC$. 
\item[(ii)]
For each  $\l \in \L$, define an abelian category $\AC_{\l}$. These we call {\em strata categories}.
\end{enumerate}
This data must satisfy the conditions
\begin{enumerate}
\item
$\AC_{\emptyset} = 0$ and $\AC_{\L} = \AC$.
\item
For each pair of lower subposets $\G' \subset \G \subset \L$, there are inclusions of Serre categories 
$ i_* \colon \AC_{\G'} \into \AC_{\G}$.
\item
For each lower subposet $\G \subset \L$, and maximal $\l \in \G$ there is a recollement
\[
\begin{tikzcd}
	{0} & {\AC_{\G \setminus \{\l\}}} & \AC_{\G} & {\AC_{\l}} & {0}
	\arrow[from=1-1, to=1-2]
	\arrow["i_*", from=1-2, to=1-3]
	\arrow["j^*", from=1-3, to=1-4]
	\arrow[from=1-4, to=1-5]
\end{tikzcd}
\]
\item
For lower subposets $\G' \subset \G \subset \L$ in which $\l \in \L$ is maximal in both $\G'$ and $\G$, there is compatibility between the adjoints of the functors $j^*\colon \AC_{\G'} \to \AC_{\l}$ and the adjoints of $j^*\colon \AC_{\G} \to \AC_{\l}$. That is, the following diagrams commute.
\begin{equation}\label{stratcompat}
\begin{tikzcd}
	{\AC_{\l}} & {\AC_{\G'}} && {\AC_{\l}} & {\AC_{\G'}} \\
	& {\AC_{\G}} &&& {\AC_{\G}}
	\arrow["{j_{!}}", from=1-1, to=1-2]
	\arrow["{j_{!}}"', from=1-1, to=2-2]
	\arrow["{i_{*}}", from=1-2, to=2-2]
	\arrow["{j_{*}}", from=1-4, to=1-5]
	\arrow["{j_{*}}"', from=1-4, to=2-5]
	\arrow["{i_{*}}", from=1-5, to=2-5]
\end{tikzcd}
\end{equation}
\end{enumerate}
\end{manualdefinition}

For example, the category $P_{\L} (X, \km)$ has a stratification by the poset $\L$ with closure order: $\l \leq \mu$ if $X_{\l} \subset \overline{X_{\mu}}$. Further examples include equivariant perverse sheaves, highest weight categories, and  $\e$-stratified categories in the sense of Brundan and Stroppel \cite{BS18}.

Recall that, for a field $\km$, a $\km$-linear abelian category $\AC$ is {\em finite over $\km$} if $\AC$ is equivalent to a category of finite dimensional modules of a finite dimensional $\km$-algebra.
This paper is inspired by a result of Cipriani and Woolf \cite[Corollary 5.2]{CW21} that says that a category of perverse sheaves (with coefficients in a field) on a space stratified by finitely many strata is finite if and only if the same is true for each category of finite type local systems on each stratum. We extend this result by showing that an abelian category with a stratification by a finite poset is finite over $\km$ if and only if the same is true of all strata categories (Corollary \ref{recolenoughprojcor1}). Moreover, we give necessary and sufficient conditions for a finite abelian category to be $\e$-stratified (Theorem \ref{eStratTheorem}). This result specializes to give necessary and sufficient conditions for a finite abelian category to be standardly stratified in the sense of Cline, Parshall, Scott \cite{CPS96}, and further specialises to recover a characterisation of highest weight categories due to Krause \cite{Kra17a}.

 \subsection{Outline and explanation of main results}
Let $\AC$ be an abelian category with a stratification by a poset $\L$. For each $\l \in \L$, define the Serre quotient functor $j^{\l} \colon \AC_{\{\mu \in \L ~|~ \mu \leq \l \}} \to \AC_{\l}$, and let $j_{!}^{\l} \colon \AC_{\l} \to \AC_{\{\mu \in \L ~|~ \mu \leq \l \}}$ and $j_{*}^{\l} \colon \AC_{\l} \to \AC_{\{\mu \in \L ~|~ \mu \leq \l \}}$ be the left and right adjoints of $j^{\l}$. By a slight abuse of notation, write $j_{!}^{\l} \colon \AC_{\l} \to \AC$ and $j_{*}^{\l} \colon \AC_{\l} \to \AC$ for the functors obtained by postcomposing with the inclusion functor $\AC_{\{\mu \in \L ~|~ \mu \leq \l \}} \into \AC$.

In Section \ref{recprel} we recall some basic features of recollements and stratifications, and describe geometric examples of stratifications of abelian categories.

In Section \ref{recintext} we define, for each $\l \in \L$, the {\em intermediate-extension functor} $j^{\l}_{!*}\colon \AC_{\l} \to \AC$:
\[
j_{!*}^{\l} X := \im(\overline{1_X}\colon  j_{!}^{\l} X \to j_{*}^{\l} X ),
\]
where $\overline{1_X}$ is the morphism corresponding to the identity $1_X$ under the natural isomorphism
\[
\Hom_{\AC} (j_{!}^{\l} X, j_{*}^{\l} X) \simeq \Hom_{\AC_{\l}} ( X, j^{\l} j_{*}^{\l} X) \simeq \Hom_{\AC_{\l}} ( X,  X) .
\]
A result of  Kuhn \cite[Proposition 4.6]{Kuh94} (see also Proposition \ref{recol4})  is that every simple object $L \in \AC$ is of the form $j^{\l}_{!*} L_{\l}$, for a unique (up to isomorphism) simple object $L_{\l} \in \AC_{\l}$ and unique $\l \in \L$. Proposition \ref{recol41} says that if $\AC$ is an abelian category with a stratification by a finite poset, then every object in $\AC$ has a finite filtration by simple objects if and only if the same is true of all the strata categories. 


In Section \ref{ringstrat} we study rings, $A$, with a complete orthogonal set of idempotents, $\{ e_{\l} \}_{\l \in \L}$, indexed by a poset $\L$, and satisfying the relations 
\begin{align}
\sum_{\mu > \l } e_{\l} A e_{ \mu} A &= \sum_{\mu \nleq \l} e_{\l} A e_{\mu} A \label{StratAlg1in} \\
\sum_{\mu > \l } A e_{ \mu} Ae_{\l} &= \sum_{\mu \nleq \l } A e_{\mu} Ae_{\l} . \label{StratAlg2in}
\end{align}
We say that such a ring is {\em stratified} by the idempotents $\{ e_{\l} \}_{\l \in \L}$. Proposition \ref{StratRingToStratCat} says that for such rings, $\AC = \rMod A$ has a stratification of abelian categories with strata categories
\[
\AC_{\l} = \rMod \frac{e_{\l} A e_{\l}}{ \sum_{\mu \nleq \l } e_{\l} A e_{ \mu} A e_{\l}} .
\]
Under this stratification, the equalities (\ref{StratAlg1in}) and (\ref{StratAlg2in}) are the ring theoretic versions of the commutative diagrams in (\ref{stratcompat}).
Proposition \ref{SemiPrim} says that if $A$ is a semiprimary ring and $\rMod A$ has a stratification as an abelian category then there is a stratifying set of idempotents,  $\{ e_{\l} \}_{\l \in \L}$, of $A$ that induces the stratification of $\rMod A$.  This generalises a result of Psaroudakis and Vit\'oria \cite[Corollary 5.5]{PV14}.

In Section \ref{rechom} we prove the main result of the paper (Theorem \ref{recolenoughproj}) that describes when the middle category in a recollement of abelian length categories has enough projectives.

\begin{manualtheorem}{\ref{recolenoughproj}}
Consider a recollement:
\[
\begin{tikzcd}
	{0} & {\AC_Z} & \AC & {\AC_U} & {0}
	\arrow[from=1-1, to=1-2]
	\arrow["{i_*}", from=1-2, to=1-3]
	\arrow["{j^*}", from=1-3, to=1-4]
	\arrow[from=1-4, to=1-5]
\end{tikzcd}
\]
Suppose $\AC_U$ and $\AC_Z$ have finitely many simple objects and every object has a finite filtration by simple objects. Suppose moreover that for any simple objects 
 $A, B$ in $\AC$,
\[
\dim_{\End_{\AC} (B)} \Ext^{1}_{\AC} (A, B) < \infty .
\]
Then $\AC$ has enough  projectives if and only if both $\AC_U$ and $\AC_Z$ have enough projectives.\footnote{Since $\End_{\AC} (B)$ is a division ring, any $\End_{\AC} (B)$-module is free. For an $\End_{\AC} (B)$-module $M$, $\dim_{\End_{\AC} (B)} M$ is the rank of $M$ as a free $\End_{\AC} (B)$-module.}
\end{manualtheorem}

A category $\AC$ is finite over $\km$ if and only if $\AC$ has finite dimensional $\Hom$-spaces, finitely many simple objects, enough projectives, and every object has finite filtration by simple objects. Such categories also have finite dimensional $\Ext$-spaces (see e.g. Proposition \ref{finExt}). In particular, the following corollary follows from 
Theorem \ref{recolenoughproj}.

\begin{manualcorollary}{\ref{recolenoughprojcor}}
Let $\AC$ be a $\km$-linear abelian category with a stratification by a finite poset. Then $\AC$ is finite over $\km$ if and only if the same is true of every strata category.
\end{manualcorollary}


To state the results in Sections \ref{standardsec} and \ref{BrunStropSection}, let $\AC$ have enough projectives and injectives, finitely many simple objects, and suppose every object in $\AC$ has finite length. Suppose $\AC$ has a stratification by a finite poset $\L$. Let $B$ be a set indexing the simple objects in $\AC$ (up to isomorphism) and write $L(b)$ (respectively $P(b)$, $I(b)$) for the simple (respectively projective indecomposable, injective indecomposable) object corresponding to $b \in B$.

Define the {\em stratification function }
\[
\rho\colon B \to \L
\]
 that maps each $b \in B$ to the corresponding $\l \in \L$ in which $L(b) = j^{\l}_{!*} L_{\l} (b)$ for some simple object $L_{\l} (b) \in \AC_{\l}$. 
 Write $P_{\l} (b)$ and $I_{\l} (b)$ for the projective cover and injective envelope of $L_{\l} (b)$ in $\AC_{\l}$. 
 
For $b \in B$ and $ \l = \r (b)$, define the {\em standard} and {\em costandard} objects
\[
\D (b) := j^{\l}_! P_{\l} (b), \qquad \nabla (b) := j^{\l}_* I_{\l} (b).
\]
%

Proposition \ref{standards} says that every projective indecomposable, $P(b)$, in $\AC$ has a filtration by quotients of standard objects, $\D (b')$, in which $\rho (b') \geq \rho (b)$. Dually, every injective indecomposable, $I(b)$, of $\AC$ has a filtration by subobjects of costandard objects, $\nabla (b')$, in which $\rho (b') \geq \rho (b)$. 

Standard and costandard objects play a crucial role in representation theoretic applications of stratifications of abelian categories, where one often requires that projective and/or injective indecomposable objects have filtrations by standard and/or costandard objects. For example, both of these conditions are required in Cline, Parshall, and Scott's definition of highest weight category \cite{CPS88b}. 

Categories whose projective indecomposables have filtrations by standard objects have been widely studied - beginning with the work of Dlab \cite{Dla96} and Cline, Parshall and Scott \cite{CPS96}. 
Categories in which both projective objects have a filtration by standard objects and injective objects have a filtration by costandard objects have been studied by various authors (see e.g. \cite{Fri07}, \cite{CZ19}, \cite{LW15}). Brundan and Stroppel \cite{BS18} (based on the work of \cite{ADL98}) define a general framework called an {\em $\e$-stratified category} that includes these situations as special cases. 
We recall this definition now. 

For a {\em sign function} $\e \colon \Lambda
\to \{+,-\}$, define
the {\em $\e$-standard} and {\em $\e$-costandard objects}
\begin{align*}
\D_\e(b) &:= \left\{
\begin{array}{ll}
j^{\l}_! P_{\l} (b)&\text{if $\e(\l) = +$}\\
j^{\l}_! L_{\l}(b)&\text{if $\e(\l) = -$}
\end{array}
\right.,&
\nabla_\e(b) &:= \left\{
\begin{array}{ll}
 j^{\l}_* L_{\l}(b)&\text{if $\e(\l) = +$}\\
 j^{\l}_* I_{\l} (b)&\text{if $\e(\l) = -$}
\end{array}
\right.,
\end{align*}
where $\l = \rho (b)$.

Brundan and Stroppel \cite{BS18} say that $\AC$  is {\em $\e$-stratified} if the following equivalent conditions hold:
\begin{enumerate}
\item
Every projective indecomposable, $P(b)$, has a filtration by $\e$-standard objects, $\D_{\e} (b')$, in which $\rho (b') \geq \rho (b)$. 
\item
Every injective indecomposable, $I(b)$, has a filtration by $\e$-costandard objects, $\nabla_{\e} (b')$, in which $\rho (b') \geq \rho (b)$.
\end{enumerate}

Note that $\AC$ is a {\em highest weight category} if and only if each $\AC_{\l}$ is semisimple and $\AC$ is $\e$-stratified for any (and all) $\e \colon \L \to \{+,-\}$.

The following result gives a criterion for a category that is finite over a field $\km$ to be $\e$-stratified.

\begin{manualtheorem}{\ref{eStratTheorem}}
Let $\AC$ be an abelian category that is finite over a field $\km$. Then $\AC$ is $\e$-stratified (for a function $\e \colon \L \to \{+,-\}$) if and only if $\AC$  has a stratification by a finite poset $\L$ in which the following conditions hold:
\begin{enumerate}
\item
For each inclusion of Serre subcategories $i_* \colon \AC_{\G} \to \AC_{\L}$, and objects $X,Y \in \AC_{\G}$, there is a natural isomorphism 
$
\Ext^{2}_{\AC_{\G}} (X,Y) \simeq  \Ext^{2}_{\AC} (i_*X,i_*Y) .
$
\item
For each $\l \in \L$,
\begin{enumerate}
\item
If $\e (\l) = +$ then $j^{\l}_* \colon \AC_{\l} \to \AC$ is exact.
\item
If $\e (\l) = -$ then $j^{\l}_! \colon \AC_{\l} \to \AC$ is exact.
\end{enumerate}
\end{enumerate}
\end{manualtheorem}

Recollements satisfying Condition (1) are known as {\em 2-homological recollements}. These are studied by Psaroudakis in \cite{Psa13}.

 Theorem \ref{eStratTheorem} has a concrete formulation for finite-dimensional algebras stratified by idempotents. More precisely, let $A$ be a finite-dimensional algebra and $A_{\leq \l} = A/\sum_{\mu \nleq \l} A e_{\mu} A$. Proposition \ref{eStratAlg} says that if $A$ is stratified by idempotents $\{e_{\l}\}_{\l \in \L}$, then the induced stratification of $\rmod A$ is $\e$-stratified if and only if the ideals $A_{\leq \l}e_{\l}A_{\leq \l} $ are projective right $A_{\leq \l} $-modules when $\e(\l)=+$, and projective left $A_{\leq \l} $-modules when $\e(\l)=-$. This result recovers the stratified algebras of type $\e$ studied in \cite{ADL98}.

One further application of Theorem \ref{eStratTheorem} is a criterion for a category of (equivariant) perverse sheaves to be $\e$-stratified - although we remark that checking these conditions (particularly the first condition) is difficult in general. 


\section{Preliminaries}\label{recprel}

We begin with a definition of recollement. The notation used in this definition will be used throughout the paper.

\begin{defn}\label{recdef}

A {\em recollement of abelian categories} consists of three abelian categories $\AC_Z$, $\AC$ and $\AC_U$ and functors:
\begin{equation}\label{recollement0}
\begin{tikzcd}
	{\AC_Z} && \AC && {\AC_U}
	\arrow["{i_*}", from=1-1, to=1-3]
	\arrow["{i^{*}}"', shift right=5, from=1-3, to=1-1]
	\arrow["{i^{!}}", shift left=5, from=1-3, to=1-1]
	\arrow["{j^*}", from=1-3, to=1-5]
	\arrow["{j_{!}}"', shift right=5, from=1-5, to=1-3]
	\arrow["{j_{*}}", shift left=5, from=1-5, to=1-3]
\end{tikzcd}
\end{equation}
satisfying the conditions:
\begin{enumerate}
\item[(R1)]
$(i^*, i_*, i^!)$ and $(j_!, j^*, j_*)$ are adjoint triples.
\item[(R2)]
The functors $i_*$, $j_!$, $j_*$ are fully-faithful. Equivalently the adjunction maps $i^* i_* \to \id \to i^! i_*$ and $j^* j_* \to \id \to j^* j_!$ are isomorphisms.
\item[(R3)]
The functors satisfy  $j^* i_* = 0$ (and so by adjunction $i^* j_! = 0 = i^! j_*$).
\item[(R4)]
The adjunction maps produce exact sequences for each object $X \in \AC$:
\begin{align}
j_! j^* X \to &X \to i_* i^* X \to 0 \label{adex1}\\
0 \to i_* i^! X \to &X \to j_* j^* X \label{adex2}
\end{align}
\end{enumerate}
Alternatively, condition (R4) can be replaced by the condition
\begin{enumerate}
\item[(R4')]
For any object $X \in \AC$, if $j^* X = 0$ then X is in the essential image of $i_*$.
\end{enumerate}

A {\em recollement of triangulated categories} is defined in the same way as a recollement of abelian categories except that condition (R4) is replaced by the existence of the triangles:
\begin{align}
j_! j^* X \to &X \to i_* i^* X \to  \\
 i_* i^! X \to &X \to j_* j^* X \to
\end{align}
for each object $X$.
%
\end{defn}

\begin{remark}
The interchangeability of (R4) and (R4') follows from the following argument. If  $j^* X = 0$ then (R4) implies that $i_* i^! X \simeq X \simeq i_* i^* X$ and so $X$ is in the essential image of $i_*$. Conversely let $\mu\colon j_! j^* \to \id$ and $\eta\colon \id \to i_* i^*$ be the adjunction natural transformations. Then there is a commutative diagram
\[
\begin{tikzcd}
	{j_! j^* X} & X & {\cok \mu_X} & 0 \\
	0 & {i_* i^*X} & {i_* i^*(\cok \mu_X)} & 0
	\arrow["{\mu_X}", from=1-1, to=1-2]
	\arrow[from=1-2, to=1-3]
	\arrow["{\eta_{j_! j^* X}}"', from=1-1, to=2-1]
	\arrow["{\eta_X}"', from=1-2, to=2-2]
	\arrow["{\eta_{\cok \mu_X}}", from=1-3, to=2-3]
	\arrow[from=2-1, to=2-2]
	\arrow[from=2-2, to=2-3]
	\arrow[from=1-3, to=1-4]
	\arrow[from=2-3, to=2-4]
\end{tikzcd}
\]
in which the rows are exact. By applying $j^*$ to the top row we see that $j^* \cok \mu_X = 0$ and so (R4') implies that $\cok \mu_X \simeq i_* i^*(\cok \mu_X) \simeq i_* i^*X$. Equation (\ref{adex2}) holds by a similar argument.
\end{remark}

Write $\AC^Z$ for the essential image of $i_*$. To reconcile Definition \ref{recdef} with the definition of recollement in Section \ref{intro}, note that by (R2), $\AC^Z \simeq \AC_Z$, and by (R4'), $\AC^Z$ is the kernel of the exact functor $j^*$ and is hence a Serre subcategory of $\AC$.

It will be useful to note that if we extend the sequences (\ref{adex1}) and (\ref{adex2}) to exact sequences
\begin{align}
0 \to K \to j_! j^* X \to &X \to i_* i^* X \to 0 \label{adex3}\\
0 \to i_* i^! X \to &X \to j_* j^* X \to K'  \to 0\label{adex4}
\end{align}
then $K$ and $K'$ are in $\AC^Z$. Indeed, by applying the exact functor $j^*$ to (\ref{adex3}) we get that $j^* K = 0$ and so $i_* i^! K \simeq K \simeq i_* i^* K$.  Likewise by applying $j^*$ to (\ref{adex4}) we get that $K' \in \AC^Z$.

Given a recollement of abelian or triangulated categories with objects and morphisms as in (\ref{recollement0}), the opposite categories form the following recollement
\begin{equation*}
\begin{tikzcd}
	{\AC_{Z}^{op}} && \AC^{op} && {\AC_{U}^{op}}
	\arrow["{i_{*}}", from=1-1, to=1-3]
	\arrow["{i^{!}}"', shift right=5, from=1-3, to=1-1]
	\arrow["{i^{*}}", shift left=5, from=1-3, to=1-1]
	\arrow["{j^{*}}", from=1-3, to=1-5]
	\arrow["{j_{*}}"', shift right=5, from=1-5, to=1-3]
	\arrow["{j_{!}}", shift left=5, from=1-5, to=1-3]
\end{tikzcd}
\end{equation*}
which we call the {\em opposite recollement}. 

The following proposition describes a useful way to characterise the functors $i^*$ and $i^!$ in any recollement.

\begin{prop}\label{recol1}
Let $\AC$ be an abelian category with a recollement as in (\ref{recollement0}). Then for any object $X \in \AC$:
\begin{enumerate}
\item[(i)]
$i_* i^{*} X$ is the largest quotient object of $X$ in $\AC^Z$.
\item[(ii)]
$i_* i^{!} X$ is the largest subobject of $X$ in $\AC^Z$.
\end{enumerate}
\end{prop}

\begin{proof}
By the adjunction $(i_*, i^*)$ and since $i_*$ is fully-faithful we have natural isomorphisms for $X \in \AC$, $Y \in \AC_Z$:
\[
\Hom_{\AC}(i_* i^* X, i_* Y)
\simeq
\Hom_{\AC_Z}(i^* X, Y)
\simeq
\Hom_{\AC} (X, i_* Y)
\]
sending $f$ to $f \circ \eta$ where $\eta \colon X \to i_* i^* X$ is the adjunction unit. In particular any morphism $X \to i_* Y$ factors through $ i_* i^* X$. Statement (i) follows. 
Statement (ii) follows by taking the opposite recollement.
\end{proof}

%

\subsection{Stratifications of abelian categories}
Say that a subset $\G$ of a poset $\L$ is {\em lower} if for any $\l \in \G$, if $\mu \leq \l$ then $\mu \in \G$.



\begin{defn}\label{stratification}
A {\em stratification} of an abelian/triangulated category $\AC$ by a non-empty poset $\L$ consists of the following data:
\begin{enumerate}
\item[(i)]
An assignment of an abelian/triangulated category $\AC_{\G}$ for every lower $\G \subset \L$, and for lower subsets $\G' \subset \G \subset \L$ a fully faithful functor 
\[
i^{\G, \G'}_{*}\colon \AC_{\G'} \into \AC_{\G}.
\]
\item[(ii)]
For each $\l \in \L$ an abelian/triangulated category $\AC_{\l}$. We call these {\em strata categories}.
\end{enumerate}
This data must satisfy the following conditions
\begin{enumerate}
\item[(S1)]
$\AC_{\emptyset} = 0$ and $\AC_{\L} =\AC$.
\item[(S2)]
For lower subsets $\G'' \subset \G' \subset \G \subset \L$, there is a natural isomorphism
\[
i^{\G, \G'}_{*} \circ i^{\G', \G''}_{*} \simeq i^{\G, \G''} \colon \AC_{\G''} \to \AC_{\G}.
\]
\item[(S3)]
For each $\l \in \L$ and lower subsets $\G\subset \L$ in which $\l \in \G$ is maximal, the functor $i^{\l, \G}_{*} := i^{\G \backslash \{\l\}, \G}_{*}\colon \AC_{\G \backslash \{\l\}} \into \AC_{\G}$ fits into a recollement 
\[
\begin{tikzcd}
	{\AC_{\G \backslash \{\l\}}} && \AC_{\G} && {\AC_{\l}}
	\arrow["{i^{\l, \G}_{*}}", from=1-1, to=1-3]
	\arrow["{i_{\l, \G}^{*}}"', shift right=5, from=1-3, to=1-1]
	\arrow["{i_{\l, \G}^{!}}", shift left=5, from=1-3, to=1-1]
	\arrow["{j_{\l, \G}^{*}}", from=1-3, to=1-5]
	\arrow["{j^{\l, \G}_{!}}"', shift right=5, from=1-5, to=1-3]
	\arrow["{j^{\l, \G}_{*}}", shift left=5, from=1-5, to=1-3]
\end{tikzcd}
\]
\item[(S4)]
If $\G' \subset \G$ are lower subsets of $\L$, and $\l \in \L$ is a maximal element of both $\G'$ and $\G$, then the following diagrams of functors commute up to natural isomorphism.

\begin{equation}\label{stratcompat1}
\begin{tikzcd}
	{\AC_{\l}} & {\AC_{\G'}} \\
	& {\AC_{\G}}
	\arrow["{j^{\l, \G'}_{!}}", from=1-1, to=1-2]
	\arrow["{j^{\l, \G}_{!}}"', from=1-1, to=2-2]
	\arrow["{i^{\G,\G'}_{*}}", from=1-2, to=2-2]
\end{tikzcd}
\end{equation}
\begin{equation}\label{stratcompat2}
\begin{tikzcd}
	{\AC_{\l}} & {\AC_{\G'}} \\
	& {\AC_{\G}}
	\arrow["{j^{\l, \G'}_{*}}", from=1-1, to=1-2]
	\arrow["{j^{\l, \G}_{*}}"', from=1-1, to=2-2]
	\arrow["{i^{\G,\G'}_{*}}", from=1-2, to=2-2]
\end{tikzcd}
\end{equation}
\end{enumerate}
\end{defn}

The following proposition shows that, under this definition, the $j^*$ functors are compatible.

\begin{prop}\label{stratcompat3}
If $\G' \subset \G$ are lower subsets of $\L$, and $\l \in \L$ is a maximal element of both $\G'$ and $\G$, then the following diagram of functors commutes up to natural isomorphism.
\[
\begin{tikzcd}
	{\AC_{\G'}} & {\AC_{\G}} \\
	& {\AC_{\l}}
	\arrow["{i^{\G,\G'}_{*}}", from=1-1, to=1-2]
	\arrow["{j_{\l, \G'}^{*}}"', from=1-1, to=2-2]
	\arrow["{j_{\l, \G}^{*}}", from=1-2, to=2-2]
\end{tikzcd}
\]
\end{prop}

\begin{proof}
We show that $j_{\l, \G'}^* \simeq j_{\l, \Gamma}^* \circ  i^{\G,\G'}_* $ using the Yoneda Lemma. For $X\in\mathcal A_\lambda$ and $Y\in\mathcal A_{\Gamma'}$, we have natural isomorphisms
\begin{align*}
\Hom_{\AC_{\l}} (X,j_{\l, \Gamma}^*  i^{\G,\G'}_* Y)
&\simeq
\Hom_{\AC_\G}
(j^{\l, \G}_{!} X,i^{\G,\G'}_*Y)\\
&\simeq
\Hom_{\AC_\G}
(i^{\G,\G'}_* j_!^{\l, \G'}X,i^{\G,\G'}_*Y)\\
&\simeq
\Hom_{\AC_{\G'}}
(j_!^{\l, \G'}X,Y)\\
&\simeq
\Hom_{\AC_{\l}}
(X,j_{\l, \G'}^* Y).
\end{align*}
Here the first and last isomorphisms follow from the adjunctions,
the second follows from (\ref{stratcompat1}), and the third follows from the
full faithfulness of $i_*$. The result follows from the Yoneda Lemma.
\end{proof}

\begin{remark}
The proof of Proposition \ref{stratcompat3} only uses the compatibility relation (\ref{stratcompat1}). A similar argument could be made that only uses the compatibility relation (\ref{stratcompat2}).
\end{remark}


\subsection{Geometric Examples}

We outline the two geometric constructions motivating our definition of stratification of abelian categories: perverse sheaves and $G$-equivariant perverse sheaves. Examples of stratifications of abelian category arising from modules of rings are described in Section \ref{ringstrat}.

\begin{ex}[Constructible sheaves with respect to a stratification]\label{recex1}
A {\em stratification} of a quasiprojective complex variety $X$ is a finite collection, $\{X_{\l} \}_{\l \in \L}$, of disjoint, smooth, connected, locally closed subvarieties, called {\em strata}, in which
$X = \bigcup_{\l \in \L} X_{\l}$ and 
for each $\l \in \L$, $\overline{X_{\l}}$ is a union of strata. In this case we equip $\L$ with the partial order
\[
\mu \leq \l  \text{ if $X_{\mu} \subset \overline{X_{\l}}$.}
\]
We use $\L$ to refer to the stratification of $X$.


For a variety $X$, let $\Loc^{ft} (X, \km)$ be the category of local systems on $X$ of finite type with coefficients in a field $\km$. Recall that, by taking monodromy, $\Loc^{ft} (X, \km)$ is equivalent to the category, $\km[\pi_1 (X)] \mod_{fd}$, of finite dimensional $\km[\pi_1 (X)]$-modules (see e.g. \cite[Theorem 1.7.9]{Ach21}).

Say that a sheaf $\FC$ on $X$ is {\em constructible} with respect to a stratification, $\L$, of $X$ if $\FC |_{X_{\l}}$ is a local system of finite type for each $\l \in \L$. Write $\DC^{b}_{\L} (X, \km)$ for the full triangulated subcategory of $\DC^{b} (X, \km)$ consisting of objects $\FC$ in which $H^k (\FC)$ is constructible with respect to $\L$.

Say that a stratification, $\L$, of $X$ is {\em good} if for any $\l \in \L$ and any object $\LC \in \Loc^{ft} (X_{\l}, \km)$, we have $j_{\l *} \LC \in \DC^{b}_{\L} (X, \km)$, where $j_{\l}\colon X_{\l} \into X$ is the embedding, and $j_{\l *}$ is the derived pushforward. Examples of good stratifications include those  satisfying the  Whitney regularity conditions \cite{Wit65}  (see e.g \cite[Remark 2.3.21]{Ach21}). 


Given a good stratification $\L$ on $X$, the triangulated category $\DC^{b}_{\L} (X, \km)$ has a stratification by $\L$ with strata categories, $\DC_{\l}$, being the full
triangulated subcategory of $D^b(X_\lambda,k)$ consisting of complexes
whose cohomology sheaves are local systems of finite type.
For lower $\G \subset \L$, the triangulated subcategories are $\DC_{\G} := \DC^{b}_{\G} (\bigcup_{\l \in \G} X_{\l})$.

For a perversity function $p\colon \L \to \ZM$, the category ${}^p P_{\L} (X, \km)$ of perverse sheaves on $X$ with respect to the stratification $\L$ (and perversity function $p$) is the full subcategory of $\DC^{b}_{\L} (X , \km)$ consisting of complexes $\FC$ in which for any stratum with inclusion $h_{\l}\colon X_{\l} \into X$:
\begin{itemize}
\item[(i)]
 $\HC^d (h_{\l}^* \FC) = 0$ if $d > p(\l)$,
\item[(ii)]
 $\HC^d (h_{\l}^! \FC) = 0$ if $d < p(\l)$,
\end{itemize}
where $\HC^d (\FC)$ refers to the $d$-th cohomology sheaf of $\FC$. The category $\AC = {}^p P_{\L} (X, \km)$ is abelian and has a stratification by $\L$, with strata categories 
\[
\AC_{\l} = \Loc^{ft}(X_{\l}, \km) [-p(\l)] \simeq \km[\pi_1 (X_{\l})] \mod_{fd}.
\]
\end{ex}

\begin{ex}[$G$-equivariant perverse sheaves]


For a complex algebraic group $G$ and quasiprojective complex $G$-variety $X$, a {\em $G$-equivariant perverse sheaf} on $X$ is, roughly speaking, a perverse sheaf on $X$ with a $G$-action compatible with the $G$-action on $X$ (see e.g. \cite[Definition 6.2.3]{Ach21} for a precise definition). The category, $P_G (X, \km)$, of $G$-equivariant perverse sheaves is the heart of a $t$-structure on the {\em $G$-equivariant derived category}, $\DC_{G}^{b} (X, \km)$ defined by Bernstein-Lunts \cite{BL94}. For a $G$-equivariant map of $G$-varieties $h\colon X \to Y$, there are equivariant versions of the four sheaf functors: $h_*, h_!, h^!, h^*$. If $i \colon Z \into X$ is the inclusion of a $G$-invariant closed subvariety with open complement $j \colon U \into X$, then there is a recollement of triangulated categories
\[
\begin{tikzcd}
	{\DC_{G}^b (Z, \km)} && \DC_{G}^b (X, \km) && {\DC_{G}^b (U, \km)}
	\arrow["{i_{*}}", from=1-1, to=1-3]
	\arrow["{i^{*}}"', shift right=5, from=1-3, to=1-1]
	\arrow["{i^{!}}", shift left=5, from=1-3, to=1-1]
	\arrow["{j^{*}}", from=1-3, to=1-5]
	\arrow["{j_{!}}"', shift right=5, from=1-5, to=1-3]
	\arrow["{j_{*}}", shift left=5, from=1-5, to=1-3]
\end{tikzcd}
\]

 If $X$ is a homogeneous $G$-variety, then every $G$-equivariant perverse sheaf is a $G$-equivariant finite type local system (shifted by $\dim_{\CM} X$). Moreover, in this case,
\begin{equation}
P_G (X , \km) \simeq \km [G^{x} / (G^{x})^{\circ}] \mod_{fd},
\end{equation}
where $G^x \subset G$ is the stabilizer of a point $x \in X$, and $(G^{x})^{\circ}$ is the connected component of $G^{x}$ containing the identity element (see e.g. \cite[Proposition 6.2.13]{Ach21} for a proof of this statement).  

Suppose $G$ acts on $X$ with finitely many orbits (each connected). Let $\L$ be a set indexing the set of $G$-orbits in $X$, and write $\OC_{\l}$ for the orbit corresponding to $\l \in \L$. Consider $\L$ as a poset with the closure order: $\l \leq \mu$ if $\OC_{\l} \subset \overline{\OC_{\mu}}$. Such a stratification is $G$-equivariantly good \cite[Exercise 6.5.2]{Ach21}. The category $\AC = P_{G} (X, \km)$ has a stratification with strata categories 
\[
\AC_{\l} = P_G (\OC_{\l} , \km) \simeq \km [G^{x_{\l}} / (G^{x_{\l}})^{\circ}] \mod_{fd},
\]
where $x_{\l} \in \OC_{\l}$.

\end{ex}

\section{The intermediate-extension functor}\label{recintext}

Consider again a recollement:
\begin{equation}\label{recollement}
\begin{tikzcd}
	{\AC_Z} && \AC && {\AC_U}
	\arrow["{i_{*}}", from=1-1, to=1-3]
	\arrow["{i^{*}}"', shift right=5, from=1-3, to=1-1]
	\arrow["{i^{!}}", shift left=5, from=1-3, to=1-1]
	\arrow["{j^{*}}", from=1-3, to=1-5]
	\arrow["{j_{!}}"', shift right=5, from=1-5, to=1-3]
	\arrow["{j_{*}}", shift left=5, from=1-5, to=1-3]
\end{tikzcd}
\end{equation}

In this section we study the full subcategory, $\AC^U \into \AC$, whose objects have no subobjects or quotients in $\AC^Z := \im i_*$. A result of Kuhn \cite[Proposition 4.6(3)]{Kuh94}
(see also Proposition \ref{recol3}) is that the restricted functor $j^* \colon \AC^U \to \AC_U$ is an equivalence of categories. The quasi-inverse $j_{!*}\colon \AC_U \to \AC^U$ is defined as follows.

\begin{defn}[$j_{!*}\colon \AC_U \to \AC^U$]\label{intextfunctor}
For an object $X \in \AC_U$, let $\overline{1_X}\colon  j_! X \to j_* X$ be the morphism corresponding to the identity on $X$ under the isomorphism
\[
\Hom_{\AC} (j_! X, j_* X) \simeq \Hom_{\AC_U} ( X, j^* j_* X) \simeq \Hom_{\AC_U} ( X,  X) .
\]
Define
\[
j_{!*} X := \im(\overline{1_X}\colon  j_! X \to j_* X ) \in \AC .
\]
\end{defn}
It is easy to check that if $X \in \AC_U$ then $j_{!*} X \in \AC^U$. Indeed as $i^! j_* X = 0$, $j_* X$ has no subobjects in $\AC^Z$. In particular, as $j_{!*} X$ is a subobject of $j_* X$ it cannot have any subobjects in $\AC^Z$. Likewise as $j_{!*} X$ is a quotient of $j_{!} X$ it cannot have any quotients in $\AC^Z$. 

We call the functor 
$
j_{!*}\colon \AC_U \to \AC
$
an {\em intermediate-extension functor}.

\begin{remark}
Not every subquotient of an object in $\AC^U$ need be in $\AC^U$. In particular, an object in $\AC^U$ may still have simple composition factors in $\AC^Z$.
\end{remark}

The following proposition is due to Kuhn \cite[Proposition 4.6(3)]{Kuh94}. We include a proof for completeness.

\begin{prop}\label{recol3}
Let $\AC$ be an abelian category with a recollement of abelian categories as in (\ref{recollement}).
Then 
$j^* \colon \AC^U \to \AC_U$ is an equivalence of categories with quasi-inverse
$j_{!*}\colon \AC_U \to \AC^U$.
\end{prop}

\begin{proof}
Since $\AC_U$ is a Serre quotient of $\AC$ by $\AC_Z$, if $X \in \AC$ has no nonzero quotient objects in $\AC^Z$, and $Y \in \AC$ has no nonzero subobjects in $\AC^Z$, then
\[
\Hom_{\AC}(X, Y) \simeq \Hom_{\AC_U} (j^* X, j^* Y) .
\]
In particular, $j^* \colon \AC^U \to \AC_U$ is fully-faithful. To show that $j^*$ is essentially surjective it suffices to show that for any object $X \in \AC_U$, $j^* j_{!*} X \simeq X$. Now, as $j^*$ is exact:
\[
j^* j_{!*} X = j^* \im (j_! X \to j_* X) \simeq \im (j^* j_! X \to j^* j_* X) \simeq \im (\id\colon X \to X) = X.
\]
The result follows.
\end{proof}

%
%

If $\AC$ has a stratification by a finite poset $\L$, then for each $\l \in \L$, there is a functor $j^{\l}_{!*}\colon \AC_{\l} \to \AC$ defined by the composition 
\[\begin{tikzcd}
	{\AC_{\l}} & {\AC_{\{ \mu \in \L ~|~ \mu \leq \l \}}} & \AC
	\arrow["{j_{!*}}", from=1-1, to=1-2]
	\arrow[hook, from=1-2, to=1-3]
\end{tikzcd}\]


The following is an immediate consequence of \cite[Proposition 4.7]{Kuh94}. 

\begin{prop}\label{recol4}
Let $\AC$ be an abelian category with a stratification by a finite poset $\L$.
Every simple object $L \in \AC$ is of the form $j^{\l}_{!*} L_{\l}$, for a unique (up to isomorphism) simple object $L_{\l} \in \AC_{\l}$ and unique $\l \in \L$.
\end{prop}

\begin{proof}
Suppose $\AC$ fits into a recollement as in (\ref{recollement}).
By Proposition \ref{recol3}, if $L \in \AC_U$ is a simple object, then $j_{!*} L$ is a simple object in $\AC$.
Moreover all the simple objects of $\AC$ are either of the form $i_{*} L$ for a simple object $L \in \AC_Z$, or of the form $j_{ !*} L$ for a simple object $L \in \AC_U$.
The statement follows via an induction argument on $|\L|$.
\end{proof}


The following properties of the intermediate-extension functor will be useful. Statement (i) of Proposition \ref{recol31} is due to Kuhn \cite[Proposition 4.8]{Kuh94}.

\begin{prop}\label{recol31}
Let $\AC$ be an abelian category with a recollement of abelian categories as in (\ref{recollement}).
Then 
\begin{enumerate}
\item[(i)]
The functor $j_{!*}\colon \AC_U \to \AC$ maps monomorphisms to monomorphisms and epimorphisms to epimorphisms.
\item[(ii)]
If $X \in \AC$ has no nonzero quotient objects in $\AC^Z$ then there is a canonical short exact sequence
\[
0 \to i_* i^! X \to X \to j_{!*} j^{*} X \to 0
\]
\item[(iii)]
If $X \in \AC$ has no nonzero subobjects in $\AC^Z$ then there is a canonical short exact sequence
\[
0 \to j_{!*} j^{*} X \to X \to i_* i^* X \to 0
\]
\item[(iv)]
Let
\[
0 \to X' \to X \to X'' \to 0
\]
be an exact sequence in $\AC_U$. The object $j_{!*} X$ has a filtration 
\[
0 = X_0 \subset X_1 \subset X_2 \subset X_3 = j_{!*} X
\]
 in which $X_1 \simeq j_{!*} X'$, $X_3/X_2 \simeq j_{!*} X''$, and $X_2 / X_1 \in \AC^Z$.
\end{enumerate}
\end{prop}

\begin{proof}
Let $f\colon X \to Y$ be a map in $\AC_U$ and define objects $K_1$, $K_2$ in $\AC$ by the exact sequence
\[
0 \to K_1 \to j_{!*} X \to j_{!*} Y \to K_2 \to 0 
\]
To prove statement (i) it suffices to show that if $j^* K_i = 0$ then $K_i = 0$. If $j^* K_i = 0$ then by (R4'), $K_i \in \AC^Z$. It follows in either case that $K_i = 0$ since $j_{!*}X$ and $j_{!*} Y$ are in $\AC^U$.

To prove statement (ii), let $X \in \AC$ have no nonzero quotients in $\AC^Z$ and consider the short exact sequence 
\[
0 \to i_* i^! X \to X \to K \to 0 
\]
Applying $i^!$ to the sequence we see that $i^! K = 0$ and so $K \in \AC^U$. So $ K \simeq j_{!*} j^* K$ and (by applying $j^*$ to this sequence) $j^* X \simeq j^* K$. Statement (ii) follows immediately. The proof of statement (iii) is similar.

To show Statement (iv), let $0 \to X' \to X \to X'' \to 0$ be an exact sequence in $\AC_U$. Let $X_1$ be the image of the monomorphism $j_{!*} X' \to j_{!*} X$ and let $X_2$ be the kernel of the epimorphism $j_{!*} X \to j_{!*} X''$. Then $X_1 \simeq j_{!*} X'$ and $X_3/X_2 \simeq j_{!*} X''$. Moreover $X_2 / X_1 \in \AC^Z$ since
\[
j^*(X_2 / X_1) \simeq j^* X_2 / j^* X_1 \simeq \ker (X \to X'') / \im (X' \to X) = 0 .
\]
\end{proof}

Say that an abelian category is a {\em length category} if every object has a finite filtration by simple objects.

\begin{prop}\label{recol41}
If $\AC$ is an abelian category with a stratification by a finite poset, then $\AC$ is a length category if and only if all the strata categories are length categories.
\end{prop}

\begin{proof}
Let $\AC$ be an abelian category fitting into a recollement of abelian categories as in (\ref{recollement}). 
It suffices to show that  $\AC$ is a length category if and only if both $\AC_Z$ and $\AC_U$ are length categories. The result follows from this statement by  induction.

Let $X$ be an object in $\AC$ and let $K$ be defined by the short exact sequence
\[
0 \to i_* i^! X \to X \to K \to 0
\]
Then $i^! K = 0$ and so applying Proposition \ref{recol31}(iii) we get the short exact sequence
\[
0 \to j_{!*} j^* K \to K \to i_* i^* K \to 0
\]
By Proposition \ref{recol31}(iv) and the exactness of $i_* \colon \AC_Z \to \AC$,  if every object in $\AC_Z$ and every object in $\AC_U$ has a finite filtration then so do both $j_{!*} j^* K$ and $i_* i^* K$. It follows that $K$ has a finite filtration and hence so does $X$.
The converse statement is obvious.
\end{proof}

\section{Stratifications of module categories}\label{ringstrat}

In this section we relate stratifications of module categories to decompositions of rings by orthogonal idempotents. Proposition \ref{StratRingToStratCat} constructs a stratification from such a decomposition, while Propositions \ref{SemiPerf} and \ref{SemiPrim} give converses for semiperfect and semiprimary rings.

For a ring $A$, let $\rMod A$ be the category of all right $A$-modules, and $\rmod A$ be the category of finitely presented right $A$-modules. Let  $e$ be an idempotent in $A$, and define the inclusion functor $i_* \colon  \rMod A / AeA \to \rMod A$. Note that $\rMod A / AeA$ is equivalent to the Serre subcategory of $\rMod A$ consisting of modules annihilated by $e$. There is a corresponding Serre quotient $j^*\colon \rMod A \to \rMod eAe $ defined
\[
j^* := \Hom_{A} (eA, -) \simeq - \otimes_A Ae .
\]
i.e. $j^* M = Me$ for any object $M \in \rMod A$.  These functors fit into a recollement of abelian categories
\begin{equation}\label{ModRecoll}
\begin{tikzcd}
	{\rMod A / AeA} &&  \rMod A && {\rMod eAe}
	\arrow["{i_{*}}", from=1-1, to=1-3]
	\arrow["{i^{*}}"', shift right=5, from=1-3, to=1-1]
	\arrow["{i^{!}}", shift left=5, from=1-3, to=1-1]
	\arrow["{j^{*}}", from=1-3, to=1-5]
	\arrow["{j_{!}}"', shift right=5, from=1-5, to=1-3]
	\arrow["{j_{*}}", shift left=5, from=1-5, to=1-3]
\end{tikzcd}
\end{equation}
where for any right $A$-module $M$:
\begin{enumerate}
\item[(i)]
$i^* M$ is the largest quotient, $N$,  of $M$ in which $Ne = 0$.
\item[(ii)]
$i^! M$ is the largest subobject, $N$, of $M$ in which $Ne = 0$.
\end{enumerate}
Moreover $j_! := - \otimes_{eAe} eA$ and $j_* := \Hom_{eAe} (Ae, -)$.

If $A$ is right artinian and $\rmod A$ has enough injective objects then the inclusion $i_* \colon \rmod A / AeA \to \rmod A$ fits into a recollement 
\[
\begin{tikzcd}
	{\rmod A / AeA} &&  \rmod A && {\rmod eAe}
	\arrow["{i_{*}}", from=1-1, to=1-3]
	\arrow["{i^{*}}"', shift right=5, from=1-3, to=1-1]
	\arrow["{i^{!}}", shift left=5, from=1-3, to=1-1]
	\arrow["{j^{*}}", from=1-3, to=1-5]
	\arrow["{j_{!}}"', shift right=5, from=1-5, to=1-3]
	\arrow["{j_{*}}", shift left=5, from=1-5, to=1-3]
\end{tikzcd}
\]
in which $j^*$ has left adjoint $j_! = - \otimes_{eAe} eA$ (see e.g \cite[Lemma 2.5]{Kra17a}). 

The following definition describes rings $A$ in which $\rMod A$ has a stratification as an abelian category.
\begin{defn}
Let $A$ be a ring with a complete set of orthogonal idempotents $\{ e_{\l}\}_{\l \in \L}$, indexed by a finite poset $\L$. Define the idempotents
\[
\e_{ >\l} = \sum_{\mu > \l} e_{\mu} \quad \text{and} \qquad \e_{\nleq \l} = \sum_{\mu \nleq \l} e_{\mu} .
\]
Say that $A$ is {\em stratified} by the idempotents $\{ e_{\l}\}_{\l \in \L}$ if for all $\l \in \L$ 
\begin{align}
e_{\l} A \e_{ >\l} A &= e_{\l} A \e_{\nleq \l} A ,\label{StratAlg1} \\
 A \e_{ >\l} Ae_{\l} &= A \e_{\nleq \l} Ae_{\l} . \label{StratAlg2}
\end{align}
\end{defn}

For example, let $Q$ be a quiver with a poset order $(\L, \leq)$, on the vertices of $Q$. Suppose $A \simeq \km Q / I$ is a quotient of the path algebra of $Q$, and $\{e_{\l }\}_{\l \in \L}$ is the set of idempotents at each vertex. Then $A$ is stratified by $\{e_{\l }\}_{\l \in \L}$ if and only if for each vertex $\l \in \L$, every path either starting or ending at $\l$ that passes through vertices $\mu \nleq \l$ is equal, in the algebra $A$, to a linear combination of paths passing through vertices $ \mu > \l$.

For example, let Q be the quiver
\[
\begin{tikzcd}
	& {e_4} & \\
	{e_2} && {e_3} \\
	& {e_1}
	\arrow["v", from=1-2, to=2-3]
	\arrow["u", from=2-1, to=1-2]
	\arrow["r"', from=2-1, to=3-2]
	\arrow["s"', from=3-2, to=2-3]
\end{tikzcd}
\]
with an order on vertices given by their height in the diagram. Let 
\[
A = \km Q / \langle rs - uv \rangle.
\]
Then, in the algebra $A$, any path between the vertices $e_2$ and $e_3$ factors through the higher vertex $e_4$.  Thus $A$ is stratified by the vertex idempotents.

\begin{prop}\label{StratRingToStratCat}
Let $A$ be a ring stratified by idempotents $\{ e_{\l}\}_{\l \in \L}$. For each $\l \in \L$ and each lower subposet $\G \subset \L$, define the rings
\[
A_{\l} = \frac{e_{\l} A e_{\l}}{ e_{\l} A \e_{\nleq \l} A e_{\l}} \quad \text{and} \qquad A_{\G} = \frac{A}{ \sum_{\mu \notin \G} A e_{\mu} A}.
\]
The category $\rMod A$ has a stratification by $\L$ with strata categories $\AC_{\l} = \rMod A_{\l}$ and Serre subcategories $\AC_{\G} = \rMod A_{\G}$.

If $A$ is right Artinian and $\rmod A$ has enough injective objects then the category $\AC' := \rmod A$ has a stratification by $\L$ with strata categories $\AC_{\l}' = \rmod A_{\l}$ and Serre subcategories $\AC_{\G}' = \rmod A_{\G}$.

\end{prop}

\begin{proof}
We show the first statement. The second statement follows by \cite[Lemma 2.5]{Kra17a}.

The stratification axioms (S1) and (S2) are obvious.

Let $\G \subset \L$ be a lower subposet and let $\l \in \G$ be a maximal element of this subposet. 
For a lower subposet $\G \subset \L$, let 
\[
I_{\G} = \sum_{\mu \notin \G} A e_{\mu} A . 
\]
By (\ref{StratAlg1}) and (\ref{StratAlg2}), the inclusions
\[
e_{\l} A \e_{> \l} A \subset e_{\l} I_{\G} \subset e_{\l} A \e_{\nleq \l} A 
\] 
and 
\[
 A \e_{> \l} A e_{\l}\subset  I_{\G} e_{\l} \subset A \e_{\nleq \l} A  e_{\l} 
\] 
are equalities. 
In particular, if $\G'$ and $\G$ are two subposets of $\L$ in which $\l \in \L$ is maximal in both, then $e_{\l} I_{\G'} = e_{\l} I_{\G}$ and $I_{\G'} e_{\l} =  I_{\G} e_{\l}$.  

Let $\tilde{e}_{\l}  \in A_{\G}$ be the idempotent $\tilde{e}_{\l} := e_{\l} + \sum_{\mu \notin \G} A e_{\mu} A$. There is a recollement 
\[
\begin{tikzcd}
	{\rMod A_{\G} / A_{\G} \tilde{e}_{\l}  A_{\G}  } && \rMod A_{\G} && {\rMod \tilde{e}_{\l} A_{\G} \tilde{e}_{\l}}
	\arrow["{i_{*}}", from=1-1, to=1-3]
	\arrow["{i^{*}}"', shift right=5, from=1-3, to=1-1]
	\arrow["{i^{!}}", shift left=5, from=1-3, to=1-1]
	\arrow["{j^{*}}", from=1-3, to=1-5]
	\arrow["{j_{!}}"', shift right=5, from=1-5, to=1-3]
	\arrow["{j_{*}}", shift left=5, from=1-5, to=1-3]
\end{tikzcd}
\]
Note that $A_{\G} / A_{\G} \tilde{e}_{\l}  A_{\G} \simeq A_{\G \setminus \{ \l \}}$ by the third isomorphism theorem. Moreover 
\[
\tilde{e}_{\l} A_{\G} \tilde{e}_{\l} \simeq e_{\l} A e_{\l} /   e_{\l} I_{\G} e_{\l} \simeq A_{\l} .
\]
Axiom (S3) follows immediately.

The functors $j_{!}^{\l, \G} \colon \rMod A_{\l} \to \rMod A_{\G}$ and $j_{*}^{\l, \G} \colon \rMod A_{\l} \to \rMod A_{\G}$ are defined
\begin{align*}
j_{!}^{\l, \G} = - \otimes_{A_{\l}} \frac{ e_{\l} A  }{   e_{\l} I_{\G}} 
\qquad
\text{and}
\qquad
j_{*}^{\l, \G} = \Hom_{A_{\l}} (\frac{A e_{\l} }{    I_{\G} e_{\l}}, -) .
\end{align*}
The compatibility relations  (\ref{stratcompat1}) and  (\ref{stratcompat2}) in axiom (S4) follow from (\ref{StratAlg1}) and (\ref{StratAlg2}) respectively.
\end{proof}

Recall that a ring $A$ is called {\em semiperfect} if $A$ has a complete set of orthogonal idempotents, $\{e_b \}_{b \in B}$, in which each ring $e_b A e_b$ is local. If $A$ is semiperfect then the simple right $A$-modules, $\{ L(b) \}_{b \in B} $, are the heads of the projective indecomposable modules $P(b) := e_b A$.  Such a ring is called {\em semiprimary} if the Jacobson radical $\Rad (A)$ is nilpotent and $A / \Rad (A)$ is semisimple.  

A result of Psaroudakis and Vit\'oria \cite[Corollary 5.4]{PV14} is that if $A$ is a semiperfect ring, and $\rMod A$ is the middle category in a recollement in which each category is a module category, then the recollement is equivalent to one of the form (\ref{ModRecoll}) for some idempotent $e \in A$. Moreover they show that if $A$ is semiprimary then any recollement of $\rMod A$ is induced by an idempotent $e \in A$ as in (\ref{ModRecoll}) \cite[Corollary 5.5]{PV14}. The following two propositions extend these results to stratifications.

\begin{prop}\label{SemiPerf}
Let $A$ be a semiperfect ring. Suppose $\AC = \rMod A$ has a stratification by a finite poset $\L$, and that every stratum category, $\AC_{\l}$, is equivalent to $\rMod A_{\l}$ for some unital ring $A_{\l}$. Then $A$ has a complete set of pairwise orthogonal idempotents $\{ e_{\l} \}_{\l\in \L}$ which stratifies $A$. Moreover the  stratification of $\rMod A$ is equivalent to the one induced by the idempotents.
\end{prop}

\begin{prop}\label{SemiPrim}
Let $A$ be a semiprimary ring. If $\AC = \rMod A$ has a stratification by a finite poset $\L$ then $A$ has a complete set of pairwise orthogonal idempotents $\{ e_{\l} \}_{\l\in \L}$ that stratify $A$. Moreover the stratification of $\rMod A$ is equivalent to the one induced by the idempotents.
\end{prop}

The proof of these propositions uses a result, Proposition \ref{recol5}, from the next section. The proof of Proposition \ref{recol5} does not depend on results in Section \ref{ringstrat}. We will briefly summarise this result here. 

Given a recollement of an abelian category $\AC$ as in (\ref{recollement}), Proposition \ref{recol5} says that if $P  \to Y$ is a projective cover of an object 
 $Y$ in $\AC_U$, then $j_! P \to j_{!*} Y$ is a projective cover in $\AC$. Here a projective cover $ P \to X$ is an epimorphism in which $P$ is projective and  the kernel is superfluous - these are described in more detail in Section \ref{ProjCov}.

\begin{proof}[Proof of Propositions \ref{SemiPerf} and \ref{SemiPrim}]
Let $A$ be a basic ring satisfying either the conditions of Proposition \ref{SemiPerf} or those of Proposition \ref{SemiPrim}. The assumption that $A$ is basic is applied to make the proof shorter - a more careful choice of idempotents (and of the rings $A_{\l}$) extends the proof to nonbasic rings.

Let $B$ be a set indexing the simple objects of $\AC$ up to isomorphism. Write $L(b)$ for the simple object corresponding to $b \in B$. By Proposition \ref{recol4}, there is a function $\r\colon B \to \L$ that assigns to each $b \in B$ the $\l \in \L$ in which $L(b) \simeq j_{!*}^{\l} L_{\l} (b)$ for a simple object $L_{\l}(b)$ in $\AC_{\l}$.

Let $e_b \in A$ be the primitive idempotent corresponding to the simple object $L(b)$. For $\l \in \L$, define the idempotent 
\[
e_{\l} = \sum_{b: \rho(b) = \l} e_b .
\] 
Let $\l \in \L$ be maximal and consider the recollement
\[
\begin{tikzcd}
	{\AC_{\L \setminus \{\l\}}} && {\rMod A} && {\rMod A_{\l}}
	\arrow["{i_{*}}", from=1-1, to=1-3]
	\arrow["{i^{*}}"', shift right=5, from=1-3, to=1-1]
	\arrow["{i^{!}}", shift left=5, from=1-3, to=1-1]
	\arrow["{j^{*}}", from=1-3, to=1-5]
	\arrow["{j_{!}}"', shift right=5, from=1-5, to=1-3]
	\arrow["{j_{*}}", shift left=5, from=1-5, to=1-3]
\end{tikzcd}
\]
If $A$ is a semiprimary ring then such a recollement exists by \cite[Corollary 5.4]{PV14}.

Assume that the ring $A_{\l}$ is basic. The right $A$-module  $e_{\l} A$ is the projective cover of $\bigoplus_{b \in \rho^{-1} (\l)} L(b)$.  A consequence of Proposition \ref{recol5} is that the projective cover of $\bigoplus_{b \in \rho^{-1} (\l)} L(b)$ is $ j_! A_{\l}$. By the uniqueness of projective covers (se e.g. \cite[Corollary 3.5]{Kra15}), there is an $A$-module isomorphism $e_{\l} A \simeq j_! A_{\l}$. In particular we get the following chain of algebra isomorphisms. 
\[
A_{\l} \simeq \End_{A_{\l}} (A_{\l}) \simeq \End_{A} (j_! A_{\l}) \simeq \End_{A} (e_{\l} A) \simeq e_{\l} A e_{\l} .
\]
Here the second isomorphism holds because $j_!$ is fully-faithful. Note also that the functor $j^* \colon \rMod A \to \rMod A_{\l} \simeq \rMod e_{\l} A e_{\l}$ agrees with the usual $j^*$ functor in the idempotent-induced recollement (\ref{ModRecoll}). Indeed
\[
j^* (-) \simeq \Hom_{A_{\l}} (A_{\l}, j^* (-)) \simeq \Hom_{A} (j_! A_{\l}, -) \simeq \Hom_{A} (e_{\l} A , -).
\]
It follows (see e.g \cite[Proposition 5.2]{PV14}) that $\AC_{\L \setminus \{\l\}} \simeq \rMod A / A e_{\l} A$ and the recollement is equivalent to the recollement induced by the idempotent $e_{\l}$.

By repeatedly applying this argument, we get that for any lower subset $\G \subset \L$,  
\[
\AC_{\G} \simeq \rMod \frac{A}{ \sum_{\mu \notin \G} A e_{\mu} A} .
\]
The stratification axiom (S4) ensures that the idempotents $\{ e_{\l}\}_{\l \in \L}$ satisfy equations (\ref{StratAlg1}) and  (\ref{StratAlg2}). Hence $A$ is indeed stratified by the idempotents $\{ e_{\l}\}_{\l \in \L}$.
\end{proof}

\section{Recollements with enough projectives/injectives}\label{rechom}

In this section we study the relationship between projective covers of objects in the different categories of a recollement. More precisely, let $\AC$ be a category fitting into a recollement as in (\ref{recollement}). Proposition \ref{recol42} says that if $\AC$ has enough projectives/injectives then so does $\AC_U$. Proposition \ref{recol43} says that if $\AC$ has enough projectives/injectives then so does $\AC_Z$. Proposition \ref{recol5} says that if $X \in \AC_U$ has a projective cover $P$ in $\AC_U$ then $j_! P$ is a projective cover in $\AC$ of $j_{!*} X$. 

Unfortunately it is not easy to find a projective cover in $\AC$ of an object $i_* X \in \AC^Z$, even if a projective cover of $X$ exists in $\AC_Z$. Theorem \ref{recolenoughproj} gives sufficient conditions for such a projective cover to exist. A consequence of Theorem \ref{recolenoughproj} (Corollary \ref{recolenoughprojcor}) is that a category $\AC$ with a stratification by a finite poset is equivalent to a category of finite dimensional modules of a finite dimensional algebra if and only if the same is true of the strata categories.

\subsection{Projective covers}\label{ProjCov}

Recall that an epimorphism $\phi \colon X \to Y$ is {\em essential} if for any morphism $\a\colon X' \to X$, if $\phi \circ \a$ is an epimorphism then $\a$ is an epimorphism. Equivalently $\phi \colon X \to Y$ is essential if its kernel, $\ker \phi$, is a superfluous subobject of $X$ i.e. for any subobject $U \subset X$, if $U + \ker \phi = X$ then $U = X$.

If $P \to X$ is an essential epimorphism and $P$ is projective then we call $P$ (or more accurately the morphism $P \to X$) a {\em projective cover} of $X$.  
The projective cover of $X$ (if it exists)  is unique up to isomorphism (see e.g. \cite[Corollary 3.5]{Kra15}).

The dual concept of an essential epimorphism is called an {\em essential extension}. If $X \to I$ is an essential extension and $I$ is injective then this extension is called the {\em injective envelope} of $X$. 

If $L \in \AC$ is a simple object and $P$ is projective then $\phi \colon P \to L$ is a projective cover if and only if the following equivalent conditions hold:
\begin{itemize}
\item[(i)]
$\ker \phi$ is the unique maximal subobject of $P$.
\item[(ii)]
The endomorphism ring of $P$ is local.
\end{itemize}
See e.g. \cite[Lemma 3.6]{Kra15}. The following additional characterisation of projective covers of simple objects will be useful.

\begin{prop}\label{KSproj}
Let $\AC$ be an abelian category. Let $P \in \AC$ be a projective object of finite length and $L \in \AC$ be a simple object. A nonzero map $P \to L$ is a projective cover if and only if for any simple object $L'$
\begin{equation}\label{dimensioncondition}
\dim_{\End_{\AC} (L')} \Hom_{\AC} (P, L') = 
\begin{cases}
1 &\text{ if $L \simeq L'$,} \\
0 &\text{ otherwise.}
\end{cases}
\end{equation}
\end{prop}

\begin{remark}
For an object $B \in \AC$, if $\End_{\AC} (B)$ is a division ring then any $\End_{\AC} (B)$-module is free. In this case we write $\dim_{\End_{\AC} (B)} M$ for the rank of an $\End_{\AC} (B)$-module $M$.
\end{remark}

\begin{proof}[Proof of Proposition \ref{KSproj}]
Let $\phi \colon P \to L$ be a nonzero morphism from a projective object $P$ to simple object $L$. Then $\ker \phi$ is the unique maximal subobject of $P$ if and only if the following two statements hold:
\begin{enumerate}
\item
$\Hom_{\AC} (P, L') = 0$ whenever $L \neq L'$.  
\item
For any nonzero morphism $\phi'\colon P \to L$, there is an isomorphism $f\colon \ker \phi \to \ker \phi'$ making the following diagram commute:
\[
\begin{tikzcd}
	{\ker \phi} & P \\
	{\ker \phi'} & P
	\arrow[from=1-1, to=1-2]
	\arrow["f"', from=1-1, to=2-1]
	\arrow["{\id_P}", from=1-2, to=2-2]
	\arrow[from=2-1, to=2-2]
\end{tikzcd}
\]
\end{enumerate}
Note that this diagram can be extended to a commutative diagram with exact rows
\begin{equation}\label{commutative}
\begin{tikzcd}
	0 & {\ker \phi} & P & L & 0 \\
	\\
	0 & {\ker \phi'} & P & L & 0
	\arrow[from=1-1, to=1-2]
	\arrow[from=1-2, to=1-3]
	\arrow["f"', from=1-2, to=3-2]
	\arrow["\phi", from=1-3, to=1-4]
	\arrow["{\id_P}"', from=1-3, to=3-3]
	\arrow[from=1-4, to=1-5]
	\arrow["{f'}", from=1-4, to=3-4]
	\arrow[from=3-1, to=3-2]
	\arrow[from=3-2, to=3-3]
	\arrow["{\phi'}"', from=3-3, to=3-4]
	\arrow[from=3-4, to=3-5]
\end{tikzcd}
\end{equation}
Moreover the existence of an isomorphism $f\colon \ker \phi \to \ker \phi'$ making the left square commute in (\ref{commutative}) is equivalent to the existence of an isomorphism $f'\colon L \to L$ making the right square of (\ref{commutative}) commute.
In particular, $(2)$ is equivalent to the statement
\begin{enumerate}
\item[$(2')$]
For any nonzero morphism $\phi'\colon P \to L$, there is an isomorphism  $f'\colon L \to L$ such that $\phi' = f' \circ \phi$.
\end{enumerate}
Statement $(2')$ is equivalent to the equation $\dim_{\End_{\AC} (L)} \Hom_{\AC} (P, L) =1$. The result follows.
\end{proof}

%

In this paper we say that an abelian category has {\em enough projectives} (resp. {\em enough injectives}) if every object has a projective cover (resp. injective envelope).

\begin{prop}\label{FinProj}
Let $\AC$ be an abelian length category. Then $\AC$ has enough projectives if and only if every simple object in $\AC$ has a projective cover.
\end{prop}

\begin{proof}
Let $\AC$ be an abelian length category in which every simple object has a projective cover.
Let $X$ be an object in $\AC$ and write $\Rad X$ for the intersection of all maximal subobjects of $X$. 

The quotient object $X / \Rad X$ is semisimple. Indeed, since $\AC$ is a length category, $\Rad X$ is the intersection of finitely many maximal subobjects $M_i \subset X$.  In particular there is a monomorphism
\[
X / \Rad X \into \bigoplus_i X/ M_i
\]
given by the diagonal map.
Since $X / \Rad X$ embeds into a semisimple object, it is itself semisimple.


Since $X / \Rad X$ is semisimple, it has a projective cover $P$.  
Since $P$ is projective the essential epimorphism $P \to X / \Rad X$ factors through $X$, defining a morphism $P \to X$ that fits into the commutative diagram
\begin{equation*}
\begin{tikzcd}
	X & {X/\Rad X} \\
	P
	\arrow[two heads, from=1-1, to=1-2]
	\arrow[from=2-1, to=1-1]
	\arrow[two heads, from=2-1, to=1-2]
\end{tikzcd}
\end{equation*}
We show that this map $P \to X$ defines a projective cover of $X$.

It is straightforward to check that $\Rad X$ is a superfluous subobject of $X$. Hence the map $X \to X / \Rad X$ is an essential epimorphism. It follows that the map $P \to X$ is an epimorphism. One can check directly that $P \to X$ is essential using the fact that the maps  $P \to X /\Rad X$ and $X \to X / \Rad X$ are essential (see e.g. \cite[Lemma 3.1]{Kra15}). The result follows.

\end{proof}

\subsection{Ext-finiteness}

For $k \in \NM$, say that an abelian category $\AC$ is {\em $\Ext^k$-finite} if for any simple objects
 $A, B$ in $\AC$, 
\[
\dim_{\End_{\AC} (B)} \Ext^{k}_{\AC} (A, B) < \infty .
\]

We remark that if $\AC$ is a $\km$-linear category, for some field $\km$, then\footnote{If $\{e_i\}_{i \in I}$ is a $\km$-basis of $\End_{\AC} (B) $ and $\{f_j\}_{j \in J}$ is a $\End_{\AC} (B)$-basis of  $\Ext^{k}_{\AC} (A,B)$, then $\{ e_i\circ   f_j\}_{(i,j) \in I \times J}$ is a $\km$-basis of  $\Ext^{k}_{\AC} (A,B) $.}

\[
\dim_{\km} \Ext^{k}_{\AC} (A,B) =  \dim_{\km} \End_{\AC} (B) \times \dim_{\End_{\AC} (B)} \Ext^{k}_{\AC} (A,B).
\]

In particular, $\AC$ is $\Ext^k$-finite whenever $\dim_{\km} \Ext^{k}_{\AC} (A,B) < \infty$ for every simple object $A, B$. The converse is true if the endomorphism ring of every simple object has finite $\km$-dimension.

The $\Ext^1$-finiteness property allows us to construct new objects using the {\em universal extension construction} as defined in \cite{Rin91} - we recall this construction in Section \ref{MainResults}. This construction is particularly useful for constructing projective generators (see e.g. \cite[Theorem 3.2.1]{BGS96}, \cite[Theorem 4.6]{CW21}). 
We require $\Ext^1$-finiteness for this purpose in Theorem \ref{recolenoughproj}.

In this section we give two results about $\Ext$-finiteness (Propositions \ref{finExt} and \ref{enoughprojandinj}) that will be needed in the discussion following Theorem \ref{recolenoughproj}.

\begin{prop}\label{finExt} 
Any abelian length category with enough projectives is $\Ext^k$-finite for every $k \in \NM$.
\end{prop}

\begin{proof}
Let $\AC$ be an abelian length category.
Let $X$ be an object in $\AC$, and let $Y$ be a simple object in $\AC$. 
Consider a projective presentation of X:
\[
0 \to K \to P \to X \to 0
\] 
Since $\Ext^{k}_{\AC} (P, Y) = 0$ there is a $\End_{\AC} (Y)$-equivariant surjection of $\Ext^{k-1}_{\AC} (K, Y)$ onto $\Ext^{k}_{\AC} (X, Y)$ for each $k > 0$. Since $\AC$ is a length category, 
\[
\dim_{\End_{\AC} (Y)} \Hom_{\AC} (X, Y) < \infty.
\]
The result follows by induction.
\end{proof}


Say that a $\km$-linear abelian category is {\em finite over $\km$} if $\AC$ is a length category with enough projectives, finitely many simple objects, and finite dimensional $\Hom$-spaces. It is well-known that $\AC$ is finite over $\km$ if and only if there is a finite-dimensional $\km$-algebra $A$ in which $\AC$ is equivalent to  the category, $A \mod$, of modules that are finite dimensional as $\km$-vector spaces. Indeed, if $\{ P_{\l} \}_{\l \in \L}$ are the projective indecomposables in $\AC$ (up to isomorphism), then $A = \End_{\AC} (\bigoplus_{\l \in \L} P_{\l})^{op}$ and $\Hom_{\AC} (\bigoplus_{\l \in \L} P_{\l}, -)\colon \AC \to A \mod$ is an equivalence of categories. Note that there is a contravariant equivalence $\Hom_{\km} (- , \km)\colon A \mod \to A^{op} \mod$. In particular, any abelian category that is finite over a field $\km$ has enough injectives.

\begin{prop}\label{enoughprojandinj}
Let $\km$ be a field and let $\AC$ be a $\km$-linear abelian category with a stratification by a finite poset in which every strata category is finite over $\km$. Then $\AC$ is $\Ext^1$-finite.
\end{prop}

\begin{proof}
By the assumptions on the strata categories, $\AC$ has finite dimensional $\Hom$-spaces.
Suppose $\AC$ has a recollement with objects and morphisms as in (\ref{recollement}).
Suppose $\AC_U$ and $\AC_Z$ have enough projectives, and $\AC_U$ has enough injectives.
By Proposition \ref{finExt}, both $\AC_Z$ and $\AC_U$ are $\Ext^1$-finite. We show that $\AC$ is $\Ext^1$-finite. It suffices to show that $\dim_{\km} \Ext_{\AC}^1 (X, Y) < \infty$ for all simple objects $X, Y$.

Since $\AC^Z$ is a Serre subcategory of $\AC$, this is true whenever $X$ and $Y$ are both in $\AC^Z$. 
Let $L \in \AC_U$ be simple and let $j_{!*} L$ have projective and injective presentations:
\begin{align*}
& 0 \to K \to j_! P \to j_{!*} L \to 0 \\
& 0 \to  j_{!*} L \to j_* I \to K' \to 0
\end{align*}
The projective presentation implies that $\Hom_{\AC} (K, Y)$ surjects onto $\Ext_{\AC}^1 (j_{!*} L, Y)$. The injective presentation implies that  $\Hom_{\AC} (Y, K')$ surjects onto $\Ext_{\AC}^1 (Y, j_{!*} L)$. It follows that $\AC$ is $\Ext^1$-finite. The result follows by an induction argument.
\end{proof}

\subsection{Main results}\label{MainResults}

This section includes our results about recollements and projective covers.

\begin{prop}\label{recol42}
Let $\AC$ be an abelian category with a recollement of abelian categories as in (\ref{recollement}).
Then
\begin{enumerate}
\item[(i)]
If $\phi\colon X \to Y$ is an essential epimorphism in $\AC$ and $i^* Y = 0$ then $i^*X = 0$ and $j^* (\phi)\colon j^* X \to j^* Y$ is an essential epimorphism.
\item[(ii)]
If $P \in \AC$ is projective and $i^* P = 0$ then $j^* P \in \AC_U$ is projective. In particular if $\phi\colon P \to X$ is a projective cover in $\AC$ and $i^* X = 0$ then $j^* (\phi)\colon j^* P \to j^* X$ is a projective cover in $\AC_U$.
\end{enumerate}
In particular, if $\AC$ has enough projectives then so does $\AC_U$.
\end{prop}

\begin{proof}
Let $\phi \colon X \to Y$ be an essential epimorphism in $\AC$ and suppose $i^* Y = 0$. To show that $i^* X = 0$ it suffices to show that the canonical map $\epsilon_X \colon j_! j^* X \to X$ is an epimorphism. This follows from the following commutative diagram since $\phi$ is essential.
\[
\begin{tikzcd}
	X && Y \\
	{j_! j^* X} && {j_! j^* Y}
	\arrow["{j_! j^* \phi}"', two heads, from=2-1, to=2-3]
	\arrow["{\epsilon_X}", from=2-1, to=1-1]
	\arrow["{\epsilon_Y}"', two heads, from=2-3, to=1-3]
	\arrow["\phi", two heads, from=1-1, to=1-3]
\end{tikzcd}
\]
Let $\a \colon X' \to j^* X$ be a morphism in $\AC_U$, in which $j^*( \phi) \circ \a\colon X' \to j^* Y$ is an epimorphism. Then $\epsilon_Y \circ j_! j^*( \phi) \circ j_! \a\colon j_! X' \to Y$ is an epimorphism and so (since $\phi$ is essential) $\epsilon_X \circ j_! \a \colon j_! X' \to X$ is an epimorphism. Hence $j^* (\epsilon_X \circ j_! \a) \simeq \a \colon X' \to j^* X$ is an epimorphism. This proves (i).


For Statement (ii) it remains to show that if $P \in \AC$ is projective and $i^*P = 0$ then the functor $\Hom_{A_U} (j^* P, -) \simeq \Hom_{\AC} (P, j_*(-))$ is exact. We proceed as in the proof of \cite[Lemma 4.3]{CW21}. Given a short exact sequence $0 \to X \to Y \to Z \to 0$ in $\AC_U$, there is an exact sequence
\[
0 \to j_* X \to j_* Y \to j_* Z \to i_* C \to 0
\]
for some $C \in \AC_Z$. Applying $\Hom_{\AC}(P,-)$ to this sequence gives the exact sequence
\[
0 \to \Hom_{\AC} (P, j_* X) \to \Hom_{\AC}(P, j_* Y) \to \Hom_{\AC} (P, j_* Z ) \to 0
\]
since $\Hom_{\AC} (P, i_* C ) \simeq \Hom_{\AC_Z} (i^* P, C) =0$.
\end{proof}

\begin{prop}\label{recol43}
Let $\AC$ be an abelian category with a recollement of abelian categories as in (\ref{recollement}). Let $Y \in \AC^Z$ be an object in the essential image of $i_* \colon \AC_{Z} \to \AC$.
Then
\begin{enumerate}
\item[(i)]
If $\phi\colon X \to Y$ is an essential epimorphism in $\AC$ then $i^* (\phi)\colon i^* X \to i^* Y$ is an essential epimorphism.
\item[(ii)]
If $P \in \AC$ is projective then $i^* P \in \AC_Z$ is projective. In particular if $\phi\colon P \to Y$ is a projective cover in $\AC$ then $i^* (\phi)\colon i^* P \to i^* Y$ is a projective cover in $\AC_Z$.
\end{enumerate}
In particular, if $\AC$ has enough projectives then so does $\AC_Z$.
\end{prop}

\begin{proof}
Let $Y \in \AC^Z$ and $\phi\colon X \to Y$ be an essential epimorphism. Since $i^*$ is a left adjoint it preserves epimorphisms and so $i^* (\phi)\colon i^* X \to i^* Y$ is an epimorphism. Suppose $\a\colon X' \to i^* X$ is a morphism in $\AC_Z$ in which the map $q := i^* (\phi) \circ \a \colon X' \to i^* Y$ is an epimorphism. To show that $\a$ is an epimorphism we form the following commutative diagram in $\AC$.
\[
\begin{tikzcd}
	& Z & \\
	X && {i_* X'} \\
	& {i_* i^* X} \\
	Y \\
	& { i_* i^* Y}
	\arrow["g"', from=1-2, to=2-1]
	\arrow["h",from=1-2, to=2-3]
	\arrow["{\eta_X}", from=2-1, to=3-2]
	\arrow["\phi"', from=2-1, to=4-1]
	\arrow["{i_* \alpha}"', from=2-3, to=3-2]
	\arrow["{i_* q}", curve={height=-12pt}, from=2-3, to=5-2]
	\arrow["{i_* i^* \phi}", from=3-2, to=5-2]
	\arrow["{\eta_Y}", from=4-1, to=5-2]
	\arrow["{\eta_{Y}^{-1}}", shift left=3, from=5-2, to=4-1]
\end{tikzcd}
\]
Here the top square is a pullback diagram. The natural transformation $\eta \colon \id \to i_* i^*$ is the unit for the $(i_*, i^*)$ adjunction.  Note that $\eta_{Y}$ is invertible since $Y \in \AC^Z$.
The morphisms $\phi$, $i_* i^* \phi$, $i_* q$ are epimorphisms by our assumptions and exactness of the functor $i_*$. The morphism $\eta_X$ is an epimorphism by the recollement axiom (R4). Since the top square is a pullback it follows that $h$ is an epimorphism. So $\phi \circ g \simeq \eta_{Y}^{-1} \circ  i_* q \circ h$ is an epimorphism. Since $\phi$ is essential, it follows that  $g\colon Z \to X$ is an epimorphism. So $i_* \a \circ h \simeq \eta_X \circ g$ is an epimorphism. Hence $i_* \a$ is an epimorphism. It follows that $\a \simeq i^* i_* \a$ is an epimorphism. Hence $i^* \phi \colon i^* X \to i^*Y$ is an essential epimorphism.

Statement (ii) follows because $i^*$ is the left adjoint of an exact functor and so preserves projective objects.
\end{proof}

%


\begin{prop}\label{recol5}
Let $\AC$ be an abelian category with a recollement of abelian categories as in (\ref{recollement}).
If $X \to Y$ is an essential epimorphism in $\AC_U$ then the composition $j_! X \to j_{!*} X \to j_{!*} Y$ is an essential epimorphism in $\AC$. In particular:
\begin{enumerate}
\item[(i)]
The canonical epimorphism $j_! X \to j_{!*}X$ is essential.
\item[(ii)]
If $P \to X$ is a projective cover of $X$ in $\AC_U$ then $j_! P \to j_! X \to j_{!*} X$ is a projective cover of  $j_{!*} X$ in $\AC$.
\end{enumerate}

%

\end{prop}

\begin{proof}
Let $\phi \colon X \to Y$ be an essential epimorphism in $\AC_U$. The map $\phi'\colon j_! X \to j_{!*} X \to j_{!*} Y$ is an epimorphism by Proposition \ref{recol31}(i). Let $\a\colon X' \to j_! X$ be a morphism in which $\phi' \circ \a$ is an epimorphism. Now, $j^* (\phi') = \phi \colon X \to Y$ and since $j^*$ is exact, $j^*(\phi' \circ \a) = \phi \circ j^* (\a) \colon j^* X' \to Y$ is an epimorphism. Since $\phi$ is essential it follows that $j^* (\a)\colon j^* X' \to X$ is an epimorphism in $\AC_U$ and so $j_! j^* (\a)\colon j_! j^* X' \to j_! X $ is an epimorphism in $\AC$. The fact that $\a$ is an epimorphism follows from the commutative triangle
\[
\begin{tikzcd}
	{j_! j^* X'} && {j_! X} \\
	{X'}
	\arrow["{j_! j^* (\a)}", two heads, from=1-1, to=1-3]
	\arrow[from=1-1, to=2-1]
	\arrow["\a"', from=2-1, to=1-3]
\end{tikzcd}
\]
in which the downward arrow is the adjunction counit.
\end{proof}

The following result holds by an almost identical argument.

\begin{prop}
The intermediate-extension functor preserves essential epimorphisms and essential extensions. 
\end{prop}

The following is the main result of this section. 

\begin{thm}\label{recolenoughproj}
Let $\AC$ be an abelian length category with finitely many simple objects, and a recollement of abelian categories as in (\ref{recollement}). If $\AC$ is $\Ext^1$-finite then $\AC$ has enough  projectives if and only if both $\AC_U$ and $\AC_Z$ have enough projectives. Dually if $\AC^{op}$ is $\Ext^1$-finite then $\AC$ has enough  injectives if and only if both $\AC_U$ and $\AC_Z$ have enough injectives.
\end{thm}

Before giving the proof of this theorem we will explain one important ingredient: the {\em universal extension} as described by Ringel \cite{Rin91}.

Let $A, B$ be objects in an abelian category $\AC$ in which $\End_{\AC} (B)$ is a division ring and $d := \dim_{\End_{\AC} (B)} \Ext_{\AC}^1(A,B) < \infty$. We form the {\em universal extension} $\EC \in \Ext^{1}_{\AC} (A, B^{\oplus d})$ by the following process. First let $E_1, \ldots, E_d \in \Ext_{\AC}^1 (A, B)$ be an $\End_{\AC} (B)$-basis. The diagonal map $\D\colon A \to A^{\oplus d}$ induces a map $\Ext_{\AC}^1 (A^{\oplus d}, B^{\oplus d}) \to \Ext^{1}_{\AC} (A, B^{\oplus d})$.
Let $\EC$ be the image of $E_1 \oplus \cdots \oplus E_d$ under this map. 
Note that the $\End_{\AC} (B)$-equivariant map $\Hom_{\AC} (B^{\oplus d}, B) \to \Ext_{\AC}^1 (A,B)$ induced by the short exact sequence
\[
0 \to B^{\oplus d} \to \EC \to A \to 0
\]
is surjective (this is easy to check on the basis of $\Ext_{\AC}^1 (A,B)$).

When $B_1, \ldots, B_n$ are objects in $\AC$ in which each ring $\End_{\AC} (B_i)$ is a division ring and $d_i := \dim_{\End_{\AC} (B_i)} \Ext_{\AC}^1(A,B_i) < \infty$, we also talk about a {\em universal extension} 
$\EC \in \Ext_{\AC}^1 (A, \bigoplus_i B_{i}^{\oplus d_i})$ constructed in the following way. Let $\EC_i \in \Ext^{1}_{\AC} (A, B_{i}^{\oplus d_i})$ be a universal extension (as defined in the previous paragraph) and define $\EC$ to be the image of $\EC_1 \oplus \cdots \oplus  \EC_n$ under the map  $\Ext_{\AC}^1 (\bigoplus_i A, \bigoplus_i B^{\oplus d_i}) \to \Ext^{1}_{\AC} (A, \bigoplus_i B^{\oplus d_i})$ induced by the diagonal map $\D\colon A \to A^{\oplus n}$. Then $\EC$ has the property that the short exact sequence 
\[
0 \to \bigoplus_i B^{\oplus d_i} \to \EC \to A \to 0
\]
induces a surjection $\Hom_{\AC} (\bigoplus_i B_{i}^{\oplus d_i}, B_j) \to \Ext_{\AC}^{1} (A, B_j)$ for each $j = 1, \ldots, n$.

Dually, if $\End_{\AC} (A)^{op}$ is a division ring and $\dim_{\End_{\AC} (A)^{op}} \Ext^1 (A,B) < \infty$ then one can form a universal extension $\EC' \in \Ext_{\AC}^1 (A^{\otimes d},B)$ using the codiagonal map $\delta\colon B^{\oplus d} \to B$ instead of the diagonal map.

\begin{proof}[Proof of Theorem \ref{recolenoughproj}]
Suppose that $\AC_U$ and $\AC_Z$ have enough projectives.

Suppose $\AC^Z$ has simple objects $L_1, \ldots, L_m$ with projective covers $\bar{P}_1, \ldots, \bar{P}_m$ in $\AC^Z$. Suppose $\AC^U$ has simple objects $L_{m+1}, \ldots, L_{m+n}$. By Proposition \ref{recol5} every simple object in $\AC^U$ has a projective cover in $\AC$.
It suffices to construct a projective cover in $\AC$ of each simple object in $\AC^Z$. This amounts to finding, for each $1 \leq t \leq m$, a projective object, $P_t$, whose unique simple quotient is $L_t$.

Fix $1 \leq t \leq m$.

{\em Step 1. Define $P_t$.}
For simple object $L_{m+k} \in \AC^U$, let $P_{m+k}$ denote its projective cover in $\AC$.
Define $Q$ to be a maximal length quotient of 
\[
P := \bigoplus_{k = 1}^{n} P_{m + k}^{\oplus \dim_{\End_{\AC} (L_{m+k})} \Ext^{1}_{\AC} (\bar{P}_t, L_{m + k})}
\]
in which there is an extension 
\begin{equation}\label{Qdef}
0 \to Q \to \EC  \to \bar{P}_t \to 0
\end{equation}
inducing an isomorphism $\Hom_{\AC} (Q, L) \simeq \Ext^{1}_{\AC} (\bar{P}_t , L)$ for each $L \in \AC^U$. 
That is $0 = \Hom_{\AC} (\bar{P}_t, L) \simeq \Hom_{\AC} (\EC, L)$ and $\Ext_{\AC}^1 (\EC, L)$ injects into $\Ext_{\AC}^1 (Q, L)$.

Let $P_t$ be any choice of such $\EC$.

{\em Step 2. $P_t$ is well-defined.}
To show that the maximal quotient $Q$ exists, we just need to find one quotient of $P$ with the required property. Then since $\AC$ is a length category  there exists a maximal length quotient with the required property.
Since $\AC$ is $\Ext^1$-finite, we can let
\[
R = \bigoplus_{k =1}^{n} L_{m + k}^{\oplus \dim_{\End_{\AC} (L_{m+k})} \Ext^{1}_{\AC} (\bar{P}_t, L_{m + k})}
\]
and form the universal extension
\[
0 \to R \to \EC  \to \bar{P}_t \to 0
\]

Since this is a universal extension it induces a surjection $\Hom_{\AC} (R, L_{m+k}) \to \Ext^{1}_{\AC} (\bar{P}_t , L_{m+k})$ for each $L_{m+k} \in \AC^U$. This map is an isomorphism since 
\[
\dim_{\End_{\AC} (L_{m+k})} \Hom_{\AC} (R, L_{m+k}) =\dim_{\End_{\AC} (L_{m+k})} \Ext^{1}_{\AC} (\bar{P}_t, L_{m+k}).
\]

{\em Step 3. $P_t $ has unique simple quotient $L_t$.}
By definition of $P_t$ and by the $(i^*, i_*)$-adjunction, for any simple object $L \in \AC$:
\[
\Hom_{\AC} (P_t, L) \simeq 
\Hom_{\AC} (\overline{P}_t, L)
\simeq
\Hom_{\AC} (\overline{P}_t, i_* i^* L)
\]
and so the only simple quotient of $P_t$ is $L_t$ with multiplicity one.

{\em Step 4. $P_t$ is projective.} We show that $\Ext^{1}_{\AC} (P_{t}, L) = 0$ for each simple $L \in \AC$. 
For any simple $L \in \AC$ there is an exact sequence
\begin{equation}\label{step4}
0 \to \Ext^{1}_{\AC} (P_{t}, L) \to \Ext^{1}_{\AC} (Q, L) \to \Ext^{2}_{\AC} (\bar{P}_t, L)
\end{equation}
Indeed if $L \in \AC^U$ then this holds because $\Hom_{\AC} (Q, L) \simeq \Ext^{1}_{\AC} (\bar{P}_t , L)$. If $L \in \AC^Z$ then (\ref{step4}) holds since $\Ext_{\AC}^{1} (\bar{P}_t, L) = 0$.

To show that $P_t$ is projective it suffices to show that the third map in (\ref{step4}) is injective for any  $L \in \AC$. Suppose for contradiction that there is a nontrivial extension
\begin{equation}\label{Q'}
0 \to L \to Q' \to Q \to 0
\end{equation}
in the kernel of this map. Then there is an object $\EC \in \AC$ fitting into the following diagram
\begin{equation}\label{exactgrid}
\begin{tikzcd}
	& 0 & 0 & 0 \\
	0 & L & {Q'} & {Q} & 0 \\
	0 & L & \EC & {P_{t}} & 0 \\
	0 & 0 & {\bar{P}_t} & {\bar{P}_t} & 0 \\
	& 0 & 0 & 0
	\arrow[from=2-1, to=2-2]
	\arrow[from=2-2, to=2-3]
	\arrow[from=2-3, to=2-4]
	\arrow[from=2-4, to=2-5]
	\arrow[from=1-4, to=2-4]
	\arrow[from=2-4, to=3-4]
	\arrow[from=3-4, to=4-4]
	\arrow[from=4-4, to=5-4]
	\arrow[from=4-3, to=5-3]
	\arrow[from=3-3, to=4-3]
	\arrow[from=2-3, to=3-3]
	\arrow[from=1-3, to=2-3]
	\arrow[from=3-1, to=3-2]
	\arrow[from=3-2, to=3-3]
	\arrow[from=3-3, to=3-4]
	\arrow[from=3-4, to=3-5]
	\arrow[from=1-2, to=2-2]
	\arrow[from=2-2, to=3-2]
	\arrow[from=3-2, to=4-2]
	\arrow[from=4-2, to=4-3]
	\arrow[from=4-3, to=4-4]
	\arrow[from=4-4, to=4-5]
	\arrow[from=4-1, to=4-2]
	\arrow[from=4-2, to=5-2]
\end{tikzcd}
\end{equation}
in which each row and column is exact.
For each $L' \in \AC^U$ the sequence (\ref{Q'}) induces an exact sequence
\[
0 \to \Hom_{\AC}(Q, L') \to \Hom_{\AC}(Q', L') \to \Hom_{\AC}(L, L') 
\]
Of course, $\Hom_{\AC}(L, L') = 0$ if $L \neq L'$. If $L = L'$ the third map must be zero. Indeed if $f\colon L \to L$ factors through the inclusion $\iota\colon L \to Q'$ via a map $g\colon Q' \to L$, then $f^{-1} \circ g \colon Q' \to L$ is a retraction of $\iota$. This contradicts the assumption that (\ref{Q'}) does not split.
Hence, for any $L' \in \AC^U$, there is an isomorphism
\begin{equation}\label{maximality}
\Hom_{\AC}(Q', L') \simeq  \Hom_{\AC}(Q, L') \simeq \Ext^{1}_{\AC} (\bar{P}_t, L') .
\end{equation}
Since $P$ is projective the quotient $P \to Q$ fits into the diagram
\[
\begin{tikzcd}
	& P \\
	{Q'} & {Q}
	\arrow[two heads, from=1-2, to=2-2]
	\arrow["\varphi"', from=1-2, to=2-1]
	\arrow[two heads, from=2-1, to=2-2]
\end{tikzcd}
\]
Now $\varphi$ cannot be an epimorphism, as, by (\ref{maximality}), this would contradict the maximality of $Q$. Thus the image of $\varphi$ is isomorphic to $Q$ and so the sequence (\ref{Q'}) splits. This is a contradiction. Hence $P_t$ is projective.
\end{proof}

\begin{cor}\label{recolenoughprojcor1}
Let $\AC$ be an abelian category with a stratification by a finite poset  in which every strata category is a length category, and has finitely many simple objects. 

If $\AC$ is $\Ext^1$-finite (respectively $\AC^{op}$ is $\Ext^1$-finite) then $\AC$ has enough projectives (respectively injectives) if and only if every strata category has enough projectives (respectively injectives).
\end{cor}

\begin{proof}

By Proposition \ref{recol41}, every category $\AC_{\G}$ (for lower $\G \subset \L$) satisfies the conditions of Theorem \ref{recolenoughproj}. So if every stratum category has enough projective then we can obtain a projective cover in $\AC$ of any simple object $j^{\l}_{!*} L$ by repeatedly applying the construction in the proof of Theorem \ref{recolenoughproj} to larger and larger Serre subcategories of $\AC$.
The converse follows from Propositions \ref{recol42}  and \ref{recol43}.
\end{proof}

The following result follows immediately from Proposition \ref{enoughprojandinj} and Corollary \ref{recolenoughprojcor1}.

\begin{cor}\label{recolenoughprojcor}
Let $\AC$ be a $\km$-linear abelian category with a stratification by a finite poset. Then $\AC$ is finite over $\km$ if and only if the same is true of every strata category.
\end{cor}

From this result we recover the following result of Cipriani-Woolf.

\begin{cor}[Corollary 5.2 of \cite{CW21}]\label{CipWoo}
Let $P_{\L} (X, \km)$ be the category of perverse sheaves (with coefficients in a field $\km$) on a quasiprojective complex variety $X$ that is constructible with respect to a good stratification, $X = \bigcup_{\l \in \L} X_{\l}$, with finitely many strata.
Then $P_{\L}(X, \km)$ is finite over $\km$ if and only if each category $\km [\pi_1 (X_{\l})] \mod_{fd}$ is finite over $\km$.
\end{cor}

For example, if $X$ has a finite stratification, $X = \bigcup_{\l \in \L} X_{\l}$, in which each $X_{\l}$ has finite fundamental group, then the category $P_{\L} (X, \km)$ is finite over $\km$.

\begin{cor}\label{enoughprojEquiv}
Let $G$ be an algebraic group and let $X$ be a $G$-variety with finitely many orbits, each connected. Let $\km$ be a field. The category, $P_{G} (X, \km)$, of $G$-equivariant perverse sheaves is finite over $\km$ if and only if for each $G$-orbit $\OC_{\l}$ and $x \in \OC_{\l}$, the category $\km [G^{x} / (G^{x})^{\circ}] \mod_{fd}$ is finite over $\km$.
\end{cor}



\section{Standard and costandard objects}\label{standardsec}

In this section we focus on abelian length categories $\AC$ with enough projectives and injectives, and admitting a stratification by a finite non-empty poset $\L$.

Let $B$ be a set indexing the simple objects of $\AC$ up to isomorphism. Write $L(b)$ for the simple object corresponding to $b \in B$, and write $P(b)$ and $I(b)$ for the projective cover and injective envelope of $L(b)$. Define the {\em stratification function} $\r\colon B \to \L$ that assigns to each $b \in B$ the $\l \in \L$ in which $L(b) = j_{!*}^{\l} L_{\l} (b)$. Let $P_{\l} (b)$ and $I_{\l} (b)$ be the projective cover and injective envelope of the simple object $L_{\l}(b)$ in $\AC_{\l}$.

For each $\l \in \L$, let $\AC_{\leq \l} := \AC_{\{ \mu \in \L ~|~ \mu \leq \l \}}$.  Write $i_{*}^{\l}\colon \AC_{\leq \l} \to \AC$ for the inclusion functor and write $i^{*}_{\l}$ and $i^{!}_{\l}$ for the left and right adjoints $\AC \to \AC_{\leq \l}$. Let $j^{\l}_{!}\colon \AC_{\l} \to \AC$ be the composition 
\[
\begin{tikzcd}
	{\AC_{\l}} & {\AC_{\leq \l}} & \AC
	\arrow["{j_{!}}", from=1-1, to=1-2]
	\arrow["{i_{*}^{\l}}", from=1-2, to=1-3]
\end{tikzcd}
\]
Define  $j^{\l}_{*}\colon \AC_{\l} \to \AC$ and  $j^{\l}_{!*}\colon \AC_{\l} \to \AC$ likewise. 

For $b \in B$, define the {\em standard} and {\em costandard} objects
\[
\D (b) := j^{\l}_! P_{\l} (b), \qquad \nabla (b) := j^{\l}_* I_{\l} (b),
\]
where $ \l = \r (b)$. The following proposition gives an alternative definition of standard and costandard objects.

\begin{prop}\label{stanchar}
Let $\AC$ be an abelian category with a stratification by a finite non-empty poset $\L$. Let $\r\colon B \to \L$ be the stratification function for $\AC$. 
Let $b \in B$ be such that $\rho(b) = \l$. If $\AC$ has enough projectives then 
\[
\D (b) \simeq i^{\l}_{*} i^{*}_{\l} P(b) .
\]
If $\AC$ has enough injectives then 
\[
\nabla (b) \simeq i^{\l}_{*} i^{!}_{\l} I(b) .
\]
\end{prop}

\begin{proof}
By Proposition \ref{recol43}, $i^{*}_{\l} P(b) $ is the projective cover of $L(b)$ in $\AC_{\leq \l}$. The first result follows from Proposition \ref{recol5} and the uniqueness of projective covers. The second result follows via the dual argument.
\end{proof}

By Proposition \ref{stanchar}, $\Delta(b)$ is the largest quotient of $P(b)$, and
$\nabla(b)$ the largest subobject of $I(b)$, whose composition
factors $L(b')$ satisfy $\rho(b')\leq\lambda$.

\begin{remark}
If $\G \subset \L$ is any lower subset in which $\l \in \G$ is maximal, and $j^{\l, \G}_{!}\colon \AC_{\l} \to \AC_{\G} \to \AC$ and  $j^{\l, \G}_{*}\colon \AC_{\l} \to \AC_{\G} \to \AC$ are the recollement functors then, if $ \l = \r (b)$ then, by axiom (S4),
\begin{equation*}\label{standcompat}
\D (b) \simeq j^{\l, \G}_{!} P_{\l} (b)  \qquad \text{and} \qquad \nabla (b) \simeq j^{\l, \G}_{*} I_{\l} (b) .
\end{equation*}
Likewise if $i^{*}_{\G} \colon \AC_{\L} \to \AC_{\G} $ and $i^{!}_{\G} \colon \AC_{\L} \to \AC_{\G} $ are the left and right adjoints of the inclusion functor $i_{*}^{\G} \colon \AC_{\G} \to \AC_{\L}$, then 
\[
\D (b) \simeq i_{*}^{\G} i^{*}_{\G} P(b) \qquad \text{and} \qquad \nabla (b) \simeq i_{*}^{\G} i^{!}_{\G} I(b) .
\]
These identifications will be silently applied in induction proofs involving standard and costandard objects.
\end{remark}


The following result further extends the relationship between projective indecomposables and standard objects, and between injective indecomposables and costandard objects, in abelian categories stratified by finite posets.

\begin{prop}\label{standards}
Let $\AC$ be an abelian length category with a stratification by a finite non-empty poset $\L$. Let $\r\colon B \to \L$ be the stratification function for $\AC$. 

If $\AC$ has enough projectives then for each $\l \in \L$ and $b \in \r^{-1} (\l)$, the projective indecomposable object $P(b) \in \AC$ fits into a short exact sequence
\[
0 \to Q (b) \to P (b) \to \D (b) \to 0
\]
in which
$Q(b)$ has a filtration by quotients of $\D( b')$ satisfying $\r(b') > \r (b)$. 

If $\AC$ has enough injectives then for each $\l \in \L$ and $b \in \r^{-1} (\l)$, the injective indecomposable object $I(b) \in \AC$ fits into a short exact sequence
\[
0 \to \nabla (b) \to I (b) \to Q' (b) \to 0
\]
in which
$Q'(b)$ has a filtration by subobjects of $\nabla( b')$ satisfying $\r(b') > \r (b)$. 

\end{prop}

\begin{proof}
We just prove the first statement by induction on $|\L|$. The base case $|\L| = 1$ is trivial.

Consider the projective cover, $P(b)$, of simple object $L(b)$ in $\AC$. If $\r (b)$ is maximal then $P(b) \simeq \D (b)$ and the result holds. 
Suppose $\r (b)$ is not maximal, and choose a maximal element $\mu \in \L$ in which $ \mu > \rho (b) $. Consider the recollement
\[
\begin{tikzcd}
	{\AC_{\L \setminus \{ \mu \}}} && \AC && {\AC_{\mu}}
	\arrow["{i_{*}}", from=1-1, to=1-3]
	\arrow["{i^{*}}"', shift right=5, from=1-3, to=1-1]
	\arrow["{i^{!}}", shift left=5, from=1-3, to=1-1]
	\arrow["{j^{*}}", from=1-3, to=1-5]
	\arrow["{j_{!}}"', shift right=5, from=1-5, to=1-3]
	\arrow["{j_{*}}", shift left=5, from=1-5, to=1-3]
\end{tikzcd}
\]
There is an exact sequence
\[
j_! j^* P(b) \to P(b) \to i_* i^* P(b) \to 0
\]

By Proposition \ref{recol43}, $i^*P(b)$ is a projective indecomposable object in $\AC_{\L \setminus \{ \mu \}}$. By the induction hypothesis, $i^* P(b)$ fits into a short exact sequence
\[
0 \to Q' (b) \to i^* P (b) \to \D (b) \to 0
\]
in which $Q' (b)$ has a filtration by quotients of standard objects, $\D( b')$, in $\AC_{\L \setminus \{ \mu \}}$ satisfying $\r(b') > \r (b)$. Since $i_* \colon \AC_{\L \setminus \{ \mu \}} \to \AC$ is an exact functor that maps standard objects to standard objects, the object $i_* i^* P(b)$ fits into a short exact sequence 
\[
0 \to i_* Q' (b) \to i_* i^* P (b) \to \D (b) \to 0
\]
and $i_* Q' (b)$ has a filtration by quotients of standard objects, $\D( b')$, in $\AC_{\L \setminus \{ \mu \}}$ satisfying $ \r(b') > \r (b)$.

We next show that $K(b) := \ker(P(b) \to i_* i^* P(b))$ has a filtration by quotients of standard objects, $\D (b')$, in which $\rho (b') = \mu$. Let $R_{\mu} (b) \to j^* P(b)$ be a projective cover in $\AC_{\mu}$. Suppose $R_{\mu} (b)$  has the following decomposition into indecomposable projective objects
\[
R_{\mu} (b) = \bigoplus_{b' : \rho (b') = \mu} P_{\mu}(b')^{\oplus \eta_{b'}}.
\]
Then $j_!  R_{\mu}(b) = \bigoplus_{b'} \D(b')^{\oplus \eta_{b'}} $ is a direct sum of standard objects. Since $j_!$ is right exact, the map $j_!  R_{\mu} (b) \to j_! j^* P(b)$ is an epimorphism. This composes to give an epimorphism $j_!  R_{\mu} (b) \to  j_! j^* P(b) \to K (b)$. It follows that $K(b)$ has a filtration by quotients of standard objects, $\D (b')$, in which $\rho (b') = \mu$.

Summarising this discussion we get the following commutative diagram in which rows and columns are exact. Here the bottom left square is a pullback square defining $Q(b)$.
\[
\begin{tikzcd}
	& 0 & 0 & 0 \\
	0 & K (b) & K (b) & 0 & 0 \\
	0 & Q (b) & P (b) & \D (b) & 0 \\
	0 & i_* Q' (b)  & i_* i^* P (b) & \D (b) & 0 \\
	& 0 & 0 & 0
	\arrow[from=2-1, to=2-2]
	\arrow[from=2-2, to=2-3]
	\arrow[from=2-3, to=2-4]
	\arrow[from=2-4, to=2-5]
	\arrow[from=1-4, to=2-4]
	\arrow[from=2-4, to=3-4]
	\arrow[from=3-4, to=4-4]
	\arrow[from=4-4, to=5-4]
	\arrow[from=4-3, to=5-3]
	\arrow[from=3-3, to=4-3]
	\arrow[from=2-3, to=3-3]
	\arrow[from=1-3, to=2-3]
	\arrow[from=3-1, to=3-2]
	\arrow[from=3-2, to=3-3]
	\arrow[from=3-3, to=3-4]
	\arrow[from=3-4, to=3-5]
	\arrow[from=1-2, to=2-2]
	\arrow[from=2-2, to=3-2]
	\arrow[from=3-2, to=4-2]
	\arrow[from=4-2, to=4-3]
	\arrow[from=4-3, to=4-4]
	\arrow[from=4-4, to=4-5]
	\arrow[from=4-1, to=4-2]
	\arrow[from=4-2, to=5-2]
\end{tikzcd}
\]
The result follows.
\end{proof}

\begin{ex}[Algebras stratified by idempotents]
Let $A$ be a finite-dimensional basic algebra over a field $\km$ that is stratified by idempotents $\{ e_{\l} \}_{\l \in \L}$. 
Let $B$ be a set indexing the simple $A$-modules and $\{e_b \}_{b \in B}$ the corresponding primitive idempotents. 
The stratification function $\rho \colon B \to \L$ gives a decomposition of idempotents $e_{\l} = \sum_{b : \rho (b) = \l} e_b$.

The algebra $A_{\l} =  e_{\l} A e_{\l} / e_{\l} A \e_{\nleq \l} A e_{\l}$ has projective and injective indecomposable modules $P_{\l} (b) = e_{b} A_{\l}$ and $I_{\l} (b) = \Hom_{\km} (A_{\l} e_{b}, \km)$. The standard objects of $\rmod A$ are defined
\[
 \D (b) 
= e_{b} A_{\l} \otimes_{A_{\l}} \frac{e_{\l} A}{ e_{\l} A \e_{\nleq \l} A}
\simeq \frac{e_{b} A}{ e_b A \e_{\nleq \l} A} .
\]
The costandard objects of $\rmod A$ are defined
\begin{align*}
\nabla (b) 
&= \Hom_{A_{\l}} ( \frac{ A e_{\l}}{ A \e_{\nleq \l} A  e_{\l}}, \Hom_{\km} (A_{\l} e_{b}, \km)) \\
& \simeq 
\Hom_{\km} (\frac{ A e_{\l}}{ A \e_{\nleq \l} A  e_{\l}} \otimes_{A_{\l}} A_{\l} e_{b}, \km) \\
&\simeq 
 \Hom_{\km} (\frac{ A e_{b}}{  A \e_{\nleq \l} Ae_b} , \km) .
\end{align*}

The short exact sequences in Proposition \ref{standards} are realised as
\[
0 \to e_b A \e_{\nleq \l} A \to e_{b} A \to \frac{e_{b} A}{ e_b A \e_{\nleq \l} A} \to 0
\]
and
\[
0 \to  \Hom_{\km} (\frac{ A e_{b}}{  A \e_{\nleq \l} Ae_b} , \km) \to  \Hom_{\km} (A e_{b} , \km) \to  \Hom_{\km} ( A \e_{\nleq \l} Ae_b, \km)\to 0
\]

%
%
%
%

\end{ex}

\begin{ex}[Perverse sheaves]\label{perverseStandards}
Let $X = \bigcup_{\l \in \L} X_{\l}$ be a stratification of a quasiprojective complex variety in which the perverse sheaves category $P_{\L}(X, \km)$ has enough projectives and injectives. 
Let $\rho\colon B \to \L$ be the stratification function for $P_{\L}(X, \km)$.
For $\l \in \L$, write $j^{\l} \colon X_{\l} \into X$ for the embedding of the corresponding stratum. 
Let ${}^{p} {\HC^{0}} \colon \DC_{\L}^{b} (X, \km) \to P_{\L} (X, \km)$ be the zero-th perverse cohomology functor. 

Consider a simple perverse sheaf, $\LC_{\l} (b)$, in $P(X_{\l},\km)$ with projective cover $\PC_{\l} (b)$ and injective envelope $\IC_{\l} (b)$. Let $\LC (b) := j_{!*}^{\l} \LC_{\l} (b)$ be the corresponding simple perverse sheaf in $P_{\L}(X, \km)$. Let $\PC(b)$ and $\IC (b)$ be the projective cover and injective envelope of $\LC (b)$ in $P_{\L}(X, \km)$.

If $\l \in \L$ is maximal then $\PC (b)$ is isomorphic to the standard object $\D (b) := {}^{p} {\HC^{0}} (j_{!}^{\l} \PC_{\l}(b))$, and $\IC (b)$ is isomorphic to the costandard object $\nabla (b) := {}^{p} {\HC^{0}} (j_{*}^{\l} \IC_{\l}(b))$.

Suppose otherwise that $\l \in \L$ is not maximal. Consider the closed and open inclusions
\[
\begin{tikzcd}
	{\overline{X_{\l}}} & X & {X \setminus \overline{X_{\l}}}
	\arrow["i", hook, from=1-1, to=1-2]
	\arrow["j"', hook', from=1-3, to=1-2]
\end{tikzcd}
\]

The projective indecomposable object $\PC (b)$ fits into an exact sequence 
\[
{{}^p \HC^0 (j_! j^* \PC (b)) } \to {\PC (b)} \to {{}^p \HC^0 (i_* i^* \PC (b))} \to 0
\]
Moreover $ {{}^p \HC^0 (i_* i^* \PC (b))} \simeq \D (b)$. Proposition \ref{standards} says that the image of the map ${{}^p \HC^0 (j_! j^* \PC (b)) } \to {\PC (b)}$ has a filtration by quotients of $\D( b')$ satisfying $\r(b') > \r (b)$. 

The injective indecomposable object $\IC (b)$ fits into an exact sequence 
\[
 0 \to  {{}^p \HC^0 (i_* i^! \IC (b))} \to {\IC (b)} \to  {{}^p \HC^0 (j_* j^* \IC (b)) }
\]

Moreover $ {{}^p \HC^0 (i_* i^! \IC (b))} \simeq \nabla (b)$. Proposition \ref{standards} says that the image of the map $ {\IC (b)} \to  {{}^p \HC^0 (j_* j^* \IC (b)) }$ has a filtration by subobjects of $\nabla( b')$ satisfying $\r(b') > \r (b)$. 
\end{ex}

%

\section{Brundan and Stroppel's $\e$-stratified categories}\label{BrunStropSection}


In this section we give necessary and sufficient conditions for an abelian category that is finite over a field $\km$ and with a stratification by a finite poset to be $\e$-stratified in the sense of Brundan and Stroppel \cite{BS18} (Theorem \ref{eStratTheorem}). In Section \ref{AlgeStrat}, we apply this result to give a criterion for a finite-dimensional algebra stratified by idempotents to be $\e$-stratified (Proposition \ref{eStratAlg}). In Section \ref{hwcsec}, we specialise Theorem \ref{eStratTheorem} to highest weight categories, recovering a characterisation originally given by Krause \cite[Theorem 3.4]{Kra17a} (Proposition \ref{hwcresult}).

Let $\AC$ be an abelian length category with finitely many simples, enough projectives and injectives, and with a stratification by a finite non-empty poset $\L$. Let $\rho\colon B \to \L$ be the stratification function corresponding to this stratification.

For $b \in B$ and $\l = \rho (b) \in \L$, define {\em proper standard} and {\em proper costandard} objects 
\[
\overline{\D} (b) := j^{\l}_! L_{\l}(b),  \qquad \overline{\nabla} (b) := j^{\l}_* L_{\l}(b) .
\]

For a {\em sign function} $\e \colon \Lambda
\to \{+, - \}$, define
the {\em $\e$-standard} and {\em $\e$-costandard objects}
\begin{align*}
\D_\e(b) &:= \left\{
\begin{array}{ll}
\Delta(b)&\text{if $\e(\rho(b)) = +$}\\
\overline\Delta(b)&\text{if $\e(\rho(b)) = -$}
\end{array}
\right.,&
\nabla_\e(b) &:= \left\{
\begin{array}{ll}
\overline\nabla(b)&\text{if $\e(\rho(b)) = +$}\\
\nabla(b)&\text{if $\e(\rho(b)) = -$}
\end{array}
\right. .
\end{align*}

Note that since $j_{!}^{\l}$ and $j_{*}^{\l}$ are fully-faithful, these objects have local endomorphism rings and are hence indecomposable.
Note also that if $\r (b) \nleq \r (b')$ then for any $\e \colon \L \to \{+,-\}$,
\begin{equation}\label{exceptional}
\Hom_{\AC} (\D_{\e} (b), \D_{\e}(b')) = 0 = \Hom_{\AC} (\nabla_{\e} (b'), \nabla_{\e} (b)) .
\end{equation}
Indeed the only simple quotient of $\D_{\e} (b)$ is $L(b)$, and all simple subobjects, $L(b'')$, of $\D_{\e}(b')$ satisfy $\r(b'') \leq \r (b')$. Likewise the only simple subobject of $\nabla_{\e} (b)$ is $L(b)$, and all simple quotients, $L(b'')$, of $\nabla_{\e} (b')$ satisfy $\r(b'') \leq \r (b')$.

The following definition is due to Brundan and Stroppel \cite{BS18}.\footnote{In Brundan and Stroppel's definition, $\e$-stratified categories need not be finite (they satisfy slightly weaker finiteness conditions).}

\begin{defn}[$\e$-stratified category]\label{BSdefn}
Let $\AC$ be an abelian category that is finite over a field $\km$. Suppose $\AC$ has a stratification by a finite poset $\L$ and let $\rho\colon B \to \L$ be the corresponding stratification function. Let $\e \colon \L \to \{+, -\}$ be a function. 
Say that $\AC$ is an {\em $\e$-stratified category} if the following equivalent conditions are satisfied:
\begin{enumerate}
\item[($\e$-S1)]
For every $b \in B$, the projective indecomposable $P(b)$ has a filtration by objects of the form $\D_{\e} (b')$, where $\r (b') \geq \r (b)$.
\item[($\e$-S2)]
For every $b \in B$, the injective indecomposable $I(b)$ has a filtration by objects of the form $\nabla_{\e} (b')$, where $\r (b') \geq \r (b)$.
\end{enumerate}
\end{defn}
The equivalence of statements ($\e$-S1) and ($\e$-S2) is shown in \cite[Theorem 2.2]{ADL98}. A proof of this fact can also be found in \cite[Theorem 3.5]{BS18}.

\begin{remark}
An abelian category $ \rmod A$ is $\e$-stratified if and only if $A$ is a {\em stratified algebra of type $\e$} in the sense of \cite{ADL98}. We explain this connection in Section \ref{AlgeStrat}.
\end{remark}

We introduce the next concept.

\begin{defn}[$\e$-exact stratification]\label{BSdefn}
For a function $\e \colon \L \to \{+,-\}$, say that a stratification of an abelian category $\AC$ by a  finite non-empty poset $\L$ is {\em $\e$-exact} if for all $\l \in \L$ the following hold:
\begin{enumerate}
\item[(i)]
If $\e (\l) = +$ then the functor $j_{*}^{\l}\colon \AC_{\l} \to \AC$ is exact.
\item[(ii)]
If $\e(\l) = -$ then the functor $j_{!}^{\l} \colon \AC_{\l} \to \AC$ is exact.
\end{enumerate}
\end{defn}

For $k \in \NM$, say that a recollement of abelian categories (as in (\ref{recollement})) is {\em $k$-homological} if for all $n \leq k$ and $X,Y \in \AC_{Z}$,
\[
\Ext^{n}_{\AC_{Z}}(X, Y) \simeq \Ext^{n}_{\AC} (i_{*} X, i_{*} Y).
\]
Say that a recollement of abelian categories is {\em homological} if it is $k$-homological for all $k \in \NM$.
A study of homological recollements is given in \cite{Psa13}.

Say that a stratification of an abelian category is {\em $k$-homological} (respectively {\em homological}) if each of the recollements in the data of the stratification is $k$-homological (respectively  homological).

The following theorem is the main result of this section.

\begin{thm}\label{eStratTheorem}
Let $\AC$ be an abelian category that is finite over a field $\km$. Then, for any finite non-empty poset $\L$ and function $\e \colon \L \to \{+, -\}$, the following statements are equivalent:
\begin{enumerate}
\item
$\AC$ has an $\e$-exact homological stratification by $\L$.
\item
$\AC$ has an $\e$-exact 2-homological stratification by $\L$.
\item
$\AC$ is $\e$-stratified.
\end{enumerate}
\end{thm}

To prove this theorem we need the following lemma.

\begin{lemma}\label{2homologicalcorollary}
Let $\AC$ be an abelian length category with finitely many simple objects, and fitting into a 2-homological recollement of abelian categories as in (\ref{recollement}).

If $\AC$ has enough projectives then for any projective object $P \in \AC$, there is a short exact sequence
\[
0 \to j_! j^* P \to P \to i_* i^* P \to 0
\]

If $\AC$ has enough injectives then for any injective object $I \in \AC$, there is a short exact sequence
\[
0 \to i_* i^! I \to I \to j_* j^* I \to 0
\]
\end{lemma}

\begin{proof}
Let $P \in \AC$ be a projective indecomposable object in $\AC$. 
If $i^* P = 0$ then $j_! j^* P \simeq P$ and the result holds.

Suppose that $i^* P \neq 0$.
Consider the exact sequence
\[
0 \to K \to j_! j^* P \to P \to i_* i^* P \to 0
\]
where $K \in \AC^Z$. Since $i^* P$ is projective in $\AC_Z$ and the recollement is 2-homological, $\Ext^{2}_{\AC} (i_* i^* P, K) = 0$. In particular, there is an object $\EC \in \AC$ and an epimorphism $q\colon \EC \to P$ that fits into the following diagram in which all rows and columns are exact.
\begin{equation}\label{Ext2}
\begin{tikzcd}
	& 0 & 0 & 0 \\
	0 & K & {j_! j^* P} & {Q} & 0 \\
	0 & K & \EC & P & 0 \\
	0 & 0 & {i_* i^* P} & {i_* i^* P} & 0 \\
	& 0 & 0 & 0
	\arrow[from=1-2, to=2-2]
	\arrow[from=2-2, to=3-2]
	\arrow[from=3-2, to=4-2]
	\arrow[from=4-2, to=5-2]
	\arrow[from=4-2, to=4-3]
	\arrow[from=4-3, to=4-4]
	\arrow[from=4-4, to=4-5]
	\arrow[from=4-1, to=4-2]
	\arrow[from=3-1, to=3-2]
	\arrow[from=3-2, to=3-3]
	\arrow["q", from=3-3, to=3-4]
	\arrow[from=3-4, to=3-5]
	\arrow[from=2-1, to=2-2]
	\arrow[from=2-2, to=2-3]
	\arrow[from=2-3, to=2-4]
	\arrow[from=2-4, to=2-5]
	\arrow[from=1-4, to=2-4]
	\arrow[from=2-4, to=3-4]
	\arrow[from=3-4, to=4-4]
	\arrow[from=4-4, to=5-4]
	\arrow[from=1-3, to=2-3]
	\arrow[from=2-3, to=3-3]
	\arrow[from=3-3, to=4-3]
	\arrow[from=4-3, to=5-3]
\end{tikzcd}
\end{equation}
Since $P$ is projective, there is a map $\a\colon P \to \EC$  making the following diagram commute.
\[
\begin{tikzcd}
	P & \EC & P \\
	& {i_*i^*P}
	\arrow["\a", from=1-1, to=1-2]
	\arrow["q", two heads, from=1-2, to=1-3]
	\arrow[two heads, from=1-1, to=2-2]
	\arrow[two heads, from=1-2, to=2-2]
	\arrow[two heads, from=1-3, to=2-2]
\end{tikzcd}
\]
Since $P$ is indecomposable, the endomorphism $q \circ \a \colon P \to P$ is either nilpotent or an automorphism. Since $i_* i^* P$ is nonzero, this map is not nilpotent, and hence is an automorphism of $P$. In particular, the first two rows of Diagram (\ref{Ext2}) split. Hence $K$ is a quotient of $j_! j^* P$. Since $j_! j^* P$ has no nonzero quotient in $\AC^Z$, $K =0$. The first result follows.

The second result holds by the dual argument.
\end{proof}

\begin{proof}[Proof of Theorem \ref{eStratTheorem}]
The implication $(1) \implies (2)$ is obvious. 

For the remainder of the proof, fix a maximal element $\l \in \L$ and recollement
\begin{equation}\label{eStratProofRec}
\begin{tikzcd}
	{\AC_{\L \setminus \{ \l \}}} && \AC && {\AC_{\l}}
	\arrow["{i_{*}}", from=1-1, to=1-3]
	\arrow["{i^{*}}"', shift right=5, from=1-3, to=1-1]
	\arrow["{i^{!}}", shift left=5, from=1-3, to=1-1]
	\arrow["{j^{*}}", from=1-3, to=1-5]
	\arrow["{j_{!}}"', shift right=5, from=1-5, to=1-3]
	\arrow["{j_{*}}", shift left=5, from=1-5, to=1-3]
\end{tikzcd}
\end{equation}

$(2) \implies (3)$. Let $P (b) \in \AC$ be a projective indecomposable object. Suppose that every projective indecomposable, $P(b')$, in $\AC_{\L \setminus \{\l\}}$ has a filtration by $\e$-standard objects, $\D_{\e} (b'')$, in which $\rho (b'') \geq \rho (b')$.
In particular $i_* i^* P(b)$ has a filtration by $\e$-standard objects, $\D_{\e} (b')$, in which $\rho (b') \geq \rho (b)$.

By Lemma \ref{2homologicalcorollary}, if the recollement (\ref{eStratProofRec}) is 2-homological then $P(b)$ fits into a short exact sequence 
\[
0 \to j_! j^* P(b) \to P(b) \to i_* i^* P(b) \to 0
\]

If $j_*$ is exact then $j^* P (b)$ is projective (since $j^*$ is the left adjoint of an exact functor) and so $j_! j^* P (b)$ is a direct sum of standard objects.
If $j_!$ is exact then $j_! j^* P (b)$ has a filtration by proper standard objects.

In particular, in either case $\e (\l) = \pm$, $P(b)$ has a  filtration by $\e$-standard objects, $\D_{\e} (b')$, in which $\rho (b') \geq \rho (b)$.
The result follows by induction on $|\L|$.

$(3) \implies (1)$. Brundan-Stroppel show that if $\AC$ is $\e$-stratified then $\AC$ has an $\e$-exact stratification \cite[Theorem 3.5]{BS18} and
\begin{equation}\label{FamousExtCondition}
\Ext^{n}_{\AC} (\D_{\e} (b), \nabla_{\e} (b')) = 0
\end{equation}
 for all $b,b' \in B$ and $n \geq 1$ \cite[Theorem 3.14]{BS18}. 

Let $\AC$ be $\e$-stratified abelian length category.
Let $P$ be a projective object in $\AC_{\L \setminus \{ \l\}}$, and $I$ be an injective object in $\AC_{\L \setminus \{ \l\}}$. Then $P$ and $i_*P$ have a filtration by $\e$-standard objects, and $I$ and $i_* I$ have a filtration by $\e$-costandard objects. Since $\AC$ is a length category, then it follows from (\ref{FamousExtCondition}) that $\Ext^{n}_{\AC}  (i_* P, i_* I) = 0$. Since this is true for any projective object $P$ and injective object $I$ in $\AC_{\L \setminus \{ \l\}}$, the recollement (\ref{eStratProofRec}) is homological (see e.g. \cite[Theorem 3.9]{Psa13}). The result follows by induction on $|\L|$.

\end{proof}

\subsection{Algebras stratified by idempotents}\label{AlgeStrat}
Consider the recollement of categories of modules of a finite dimensional algebra $A$.
\[
\begin{tikzcd}
	{\rmod A / AeA} &&  \rmod A && {\rmod eAe}
	\arrow["{i_{*}}", from=1-1, to=1-3]
	\arrow["{i^{*}}"', shift right=5, from=1-3, to=1-1]
	\arrow["{i^{!}}", shift left=5, from=1-3, to=1-1]
	\arrow["{j^{*}}", from=1-3, to=1-5]
	\arrow["{j_{!}}"', shift right=5, from=1-5, to=1-3]
	\arrow["{j_{*}}", shift left=5, from=1-5, to=1-3]
\end{tikzcd}
\]
This recollement is 2-homological if and only if the multiplication map 
\[
Ae \otimes_{eAe} eA \to AeA
\]
  is an isomorphism \cite[Proposition 1.3, Lemma 1.4]{APT92}. The functor   $j_* := \Hom_{eAe} (Ae, -)$ is exact if and only if $Ae$ is a projective right $eAe$-module. The functor $j_! := - \otimes_{eAe} eA$ is exact if and only if $eA$ is a projective left $eAe$-module.
A result of \cite[Section 2.1]{CPS96} is that $Ae$ is a projective right $eAe$-module (respectively $eA$ is a projective left $eAe$-module) and $Ae \otimes_{eAe} eA \simeq AeA$ if and only if $AeA$ is a projective right (respectively left) $A$-module. 

The following result follows from this discussion and Theorem \ref{eStratTheorem}.

\begin{prop}\label{eStratAlg}
Let $A$ be a finite dimensional algebra stratified by idempotents $\{ e_{\l} \}_{\l \in \L}$, and let $A_{\leq \l} = A/\sum_{\mu \nleq \l} A e_{\mu} A$ and $A_{\l} = e_{\l}A_{\leq \l} e_{\l}$. The stratification of $\rmod A$ given by the idempotents $\{ e_{\l }\}_{\l \in \L}$ is $\e$-stratified if and only if for every $\l \in \L$
\begin{enumerate}
\item
If $\e(\l) = +$ then $A_{\leq \l}e_{\l}A_{\leq \l}$ is a projective right $A_{\leq \l}$-module.
\item
If $\e(\l) = -$ then $A_{\leq \l}e_{\l}A_{\leq \l}$ is a projective left $A_{\leq \l}$-module.
\end{enumerate}
\end{prop}

Algebras satisfying conditions (1) and (2) of Proposition \ref{eStratAlg} are studied in \cite{ADL98} where they are called {\em stratified algebras of type $\e$}. In the case that $\e = +$ is the constant function we recover the {\em standardly stratified algebras} of \cite{CPS96} and \cite{Dla96}. If the algebras $A_{\l}$ are semisimple, then the ideals $A_{\leq \l}e_{\l}A_{\leq \l}$ are heredity and the algebra $A$ is quasi-hereditary \cite{CPS88b}.
Our result provides a slight upgrade to these theories in that instead of only considering chains of ideals, our algebras contain a family of ideals indexed by a finite poset.

\subsection{Highest weight categories}\label{hwcsec}

Let $\km$ be a field. Say that a $\km$-linear abelian category $\AC$ is a {\em highest weight category} with respect to a finite poset $\L$ if $\AC$ is finite over $\km$, and for every $\l \in \L$ there is a projective indecomposable, $P_{\l}$, that fits into a short exact sequence in $\AC$:
\[
0 \to U_{\l} \to P_{\l} \to \D_{\l} \to 0
\]
in which the following hold:
\begin{enumerate}
\item[(HW1)]
$\End_{\AC} (\D_{\l})$ is a division ring for all $\l \in \L$.
\item[(HW2)]
$\Hom_{\AC} (\D_{\l}, \D_{\mu}) = 0$ whenever $\l \nleq \mu$.
\item[(HW3)]
$U_{\l}$ has a filtration by objects $\D_{\mu}$ in which $\l < \mu$.
\item[(HW4)]
$\bigoplus_{\l \in \L} P_{\l}$ is a projective generator of $\AC$.
\end{enumerate}

The following characterisation of highest weight categories is shown by Krause \cite[Theorem 3.4]{Kra17a}. We give a new proof using Theorem \ref{eStratTheorem}.

\begin{prop}\label{hwcresult}
Let $\km$ be a field, and let $\AC$ be a $\km$-linear abelian category. The following statements are equivalent.
\begin{enumerate}
\item[(1)]
$\AC$ is a highest weight category.
\item[(2)]
$\AC$ has a homological stratification with respect to a finite poset $\L$  in which every strata category is equivalent to $\rmod \G_{\l}$ for some finite dimensional division algebra $\G_{\l}$.
\item[(3)]
$\AC$ has a 2-homological stratification with respect to a finite poset $\L$  in which every strata category is equivalent to $\rmod \G_{\l}$ for some finite dimensional division algebra $\G_{\l}$.
\end{enumerate}
\end{prop}

\begin{proof}
$(1) \implies (2)$.
If $\AC$ is a highest weight category with respect to $\L$, then a homological stratification of $\AC$ by $\L$ is constructed as follows: For a lower subposet $\G \subset \L$ define $\AC_{\G}$ to be the Serre subcategory of $\AC$ generated by the standard objects $\D_{\l}$ in which $\l \in \G$. Then for any maximal $\mu \in \G$ there is a homological recollement of abelian categories
\[
\begin{tikzcd}
	{\AC_{\G \setminus \{\mu\}}} && \AC_{\G} && {\rmod \End_{\AC} (\D_{\mu})}
	\arrow["{i_{*}}", from=1-1, to=1-3]
	\arrow["{i^{*}}"', shift right=5, from=1-3, to=1-1]
	\arrow["{i^{!}}", shift left=5, from=1-3, to=1-1]
	\arrow["{ j^{*}}", from=1-3, to=1-5]
	\arrow["{j_{!}}"', shift right=5, from=1-5, to=1-3]
	\arrow["{j_{*}}", shift left=5, from=1-5, to=1-3]
\end{tikzcd}
\]
in which $j^* = \Hom_{\AC_{\G}} (\D_{\mu}, -)$ (see the proof of \cite[Theorem 3.4]{Kra17a}).

$(2) \implies (3).$ This is obvious.

$(3) \implies (1).$
 Suppose $\AC$ has a 2-homological stratification with strata categories of the form $\AC_{\l} = \rmod \G_{\l}$ for a finite dimensional division ring $\G_{\l}$. Then $\AC$ is finite (by Corollary \ref{recolenoughprojcor}).
Let $L_{\l}$ denote the unique simple object in $\AC_{\l}$. Define $\D_{\l} := j_{!}^{\l} L_{\l}$ and let $P_{\l}$ be the projective cover of $j_{!*}^{\l} L_{\l}$ in $\AC$. Statement (HW1) holds since $j_!$ is fully-faithful, Statement (HW2) is exactly equation (\ref{exceptional}), Statement (HW3) follows from Theorem \ref{eStratTheorem} (since each $j_{!}^{\l}$ is exact), and Statement (HW4) is obvious. 
\end{proof}

{\em Acknowledgements}.
This article revises and extends part of the author's PhD thesis \cite{Wig23}. This paper has benefited from discussions with Oded Yacobi, Kevin Coulembier, and helpful suggestions from an examiner of \cite{Wig23} and the referee of this paper. This work was supported by Australian Research Council grant DP190102432.

\bibliographystyle{alpha}
\bibliography{bibliography}

\end{document}